\documentclass[11pt,reqno]{amsart}

\usepackage[american]{babel}
\usepackage{comment}
\usepackage{amsmath}
\usepackage{dsfont,mathtools,amssymb}
\usepackage{color}
\usepackage[hidelinks]{hyperref}
\mathtoolsset{showonlyrefs,showmanualtags}

\newtheorem{theorem}{Theorem}[section]
\newtheorem{proposition}[theorem]{Proposition}
\newtheorem{lemma}[theorem]{Lemma}
\newtheorem{corollary}[theorem]{Corollary}
\newtheorem{remark}[theorem]{Remark}

\numberwithin{equation}{section}
\numberwithin{figure}{section}

\newcommand{\E}{\mathds{E}}
\newcommand{\Prob}{\mathds{P}}
\newcommand{\Prw}{{\mathbf P}^{\scriptscriptstyle{\rm  RW}}}

\newcommand{\Purw}{{\mathbf P}^{\scriptscriptstyle{\rm  CRW}}}

\newcommand{\Lrw}{L^{\scriptscriptstyle{\rm RW}}}
\newcommand{\cErw}{\cE^{\scriptscriptstyle \rm RW}}

\newcommand{\R}{\mathbb{R}}
\newcommand{\N}{\mathbb{N}}
\newcommand{\IND}{{\bf 1}}
\newcommand{\dd}{{\rm d}}

\DeclareMathOperator{\diag}{diag}

\newcommand{\be}{\begin{equation}}

\newcommand{\cC}{\ensuremath{\mathcal C}} 
 
\newcommand{\cE}{\ensuremath{\mathcal E}} 
\newcommand{\cF}{\ensuremath{\mathcal F}}

\newcommand{\cL}{\ensuremath{\mathcal L}}

\newcommand{\cR}{\ensuremath{\mathcal R}} 
\newcommand{\cS}{\ensuremath{\mathcal S}}

\newcommand{\bbN}{{\ensuremath{\mathbb N}} }

\newcommand{\si}{\sigma} 
 
\let\a=\alpha \let\b=\beta     
      \let\k=\kappa  \let\l=\lambda

\def\({\left(}
\def\){\right)}
\allowdisplaybreaks

\title[Cutoff for the Averaging process]{Cutoff for the Averaging process\\ on the hypercube and complete bipartite graphs}
 \subjclass[2020]{Primary 60K35; secondary 82B20; 82C26.}

 \keywords{mixing of Markov chains; cutoff phenomenon; averaging process} 

\author{Pietro Caputo, Matteo Quattropani,  and Federico Sau}
\address{Pietro Caputo\\ Universit\`{a} Roma Tre}
\email{pietro.caputo@uniroma3.it}
\address{Matteo Quattropani\\ Sapienza Universit\`{a} di Roma}
\email{matteo.quattropani@uniroma1.it}
\address{Federico Sau\\ Institute of Science and Technology Austria (ISTA)}
\email{federico.sau@ist.ac.at}

\begin{document}
\maketitle
\begin{abstract} We consider the averaging process on a graph, that is the evolution of a mass  distribution undergoing repeated averages along the edges of the graph at the arrival times of independent Poisson processes. We establish cutoff phenomena for both the $L^1$ and $L^2$ distance from stationarity when the graph is a discrete hypercube and when the graph is complete bipartite. Some general facts about the averaging process on arbitrary graphs are also discussed.   
\end{abstract}
\section{Introduction}
The averaging process on a finite graph $G=(V,E)$, referred to as ${\rm Avg}(G)$ below, is the stochastic process defined as follows. Start with some  probability mass function $\eta$ on $V$ and consider independent Poisson clocks with rate $1$ on each edge $xy\in E$; when $xy$ rings, update both masses $\eta(x)$ and $\eta(y)$  to their average value $\frac12(\eta(x)+\eta(y))$.  As soon as the underlying graph is connected, whatever the initial distribution $\eta$, the dynamics converges to the uniform distribution $\pi\equiv1/|V|$, and just as for the simple random walk on $G$, the convergence rate associated to this process is an important feature of the underlying graph. 
Unlike the case of simple random walks however,  the information available in the literature on  convergence rates for ${\rm Avg}(G)$ is rather limited. The mathematical analysis of the averaging  process started about a decade ago with the work of Aldous and Lanoue \cite{aldous_lecture_2012}. More recently, for the special case of the complete graph a full account  was given in \cite{chatterjee2020phase}, where the authors proved the existence of a cutoff phenomenon. On the contrary,  in \cite{quattropani2021mixing} the process was shown  to have no cutoff on ``finite-dimensional" graphs such as the $d$-dimensional grid, that is graphs satisfying a suitable Nash inequality which guarantees the existence of a well-defined limiting heat flow.
Very little is known for other families of graphs.  We refer to the recent work \cite{movassagh_repeated2022}  for some partial results, and for some motivating connections to 
opinion dynamics, random distributed algorithms, and quantum computing.
 
 In this paper we study the convergence to stationarity of  ${\rm Avg}(G)$ when $G=\{0,1\}^d$ is the hypercube graph with $n=2^d$ vertices, and when $G=K_{m,n-m}$ is the complete bipartite graph with the two parts of size $m$ and $n-m$, respectively. 
 In both cases, we are ultimately interested in quantifying, as $n\to \infty$, the (mean) $L^p$-distance to equilibrium for $p=1,2$, that is 
 \begin{align}\label{eq:tfun}
 	(t,\xi)\longmapsto  \color{black}	\E_\xi\left[\left\|\frac{\eta_t}{\pi}-1\right\|_{p}^p \right]^{\frac{1}{p}}\ ,
 \end{align}
 where $\eta_t$ denotes the 
 mass distribution of ${\rm Avg}(G)$ at time $t\ge 0$, while  $\E_\xi$ stands for the expectation over 
 the random updates when starting from the configuration $\eta_0=\xi$,
and  we write $\|\cdot\|_p$ for 	the $L^p(\pi)$-norm with respect to $\pi$, the uniform distribution on the vertex set $V$. For $p\geq1$, convexity  shows that for any fixed $\xi$ the function above is monotone nonincreasing in $t\geq 0$. When $p=1$, \eqref{eq:tfun} is twice the expected value of the usual total variation distance between $\eta_t$ and $\pi$. 
For worst-case initial conditions $\xi$, we analyze the distance to equilibrium at times of the form
 \begin{align}T(a)=t_{\rm mix}+aw\, ,\qquad a\in\R\, ,
 	\end{align}
 for some appropriate choices of $t_{\rm mix}=t_{\rm mix}(n)$ and $w=w(n)$ such that $t_{\rm mix}\gg w$, and prove \emph{cutoff} results around the mixing time $t_{\rm mix}$, with $w$ controling the size of the so-called \emph{cutoff window}. We refer e.g.\ to \cite{levin2017markov} for background on the cutoff phenomenon for Markov chains.

For the hypercube we establish that both the $L^1$ and $L^2$ distances have a cutoff phenomenon with window $w=O(1)$, and  mixing time $t_{\rm mix}=\frac{1}{2}\log d$, which coincides with the $L^p$-mixing time of the lazy 
random walk on  $\{0,1\}^d$ where the $d$ coordinates flip independently at rate $\frac{1}{2}$; see, e.g., \cite{levin2017markov}.
This settles a question that was mentioned as an open problem in \cite{aldous_lecture_2012}.  
 \begin{theorem}[{$G=\{0,1\}^d$}: Cutoff in $L^p$, {$p=1,2$}]
 \label{th:hypercube}
 Let $n=2^d$,
and
\begin{align}\label{eq:T-hypercube}T(a):=\frac{1}{2}\log d+a\ ,\qquad a\in\R\ .	
\end{align}
  Then, 
 \begin{equation}
 	\lim_{a\to\infty}\limsup_{n\to\infty}\sup_{\xi}	\E_\xi\left[\left\|\frac{\eta_{T(a)}}{\pi}-1\right\|^p_{p}\right] =0\ ,\qquad p=1,2\ ,
 \end{equation}
 and
 \begin{equation}
 	\lim_{a\to-\infty}\liminf_{n\to\infty}\sup_{\xi}	\E_\xi\left[\left\|\frac{\eta_{T(a)}}{\pi}-1\right\|^p_{p}\right] =\begin{cases}
 		2&\text{if }p=1\\
 		+\infty&\text{if }p=2\ .
 	\end{cases}
 \end{equation}
 \end{theorem} 
 In the case of complete bipartite graphs $G=K_{m,n-m}$, we establish two cutoff results in $L^1$ and $L^2$, with distinct locations and different sizes for the cutoff windows.  Our results apply to all sizes $1\le m\le n/2$. The two special cases $m=1$ (star graph) and $m=n/2$ (regular case) were discussed as open problems in \cite{movassagh_repeated2022}. 
 \begin{theorem}[{$G=K_{m,n-m}$}: Cutoff in $L^1$]
 \label{th:K1}
 Let, for all $1\le m\le n/2$,
 \begin{align}\label{eq:T-K1} T(a):=\frac{n}{2(n-m)}\frac{\log_2n}{m}+ a\frac{\sqrt{\log n}}{m}\ ,\qquad a\in\R\ .
\end{align}
 Then, 
\begin{equation}
\lim_{a\to\infty}\limsup_{n\to\infty}\sup_{\xi}	\E_\xi\left[\left\|\frac{\eta_{T(a)}}{\pi}-1\right\|_{1}\right] =0 \ ,
	\end{equation}
and
\begin{equation}
	\lim_{a\to-\infty}\liminf_{n\to\infty}\sup_{\xi}	\E_\xi\left[\left\|\frac{\eta_{T(a)}}{\pi}-1\right\|_{1}\right] =2\ .
\end{equation}
 \end{theorem} 

Concerning the $L^2$-norm, besides determining the size of the cutoff window, we also identify  the so-called \emph{cutoff profile}.

\begin{theorem}[{$G=K_{m,n-m}$}: Cutoff in $L^2$]
 \label{th:K2}
 Let, for all $1\le m\le n/2$,
 \begin{align}
	\label{eq:T-K2}
 T(a):=\frac{\log n}{\theta m}+ \frac{a}{\theta m}\ ,\qquad a\in\R\ ,
\end{align}
 where 
 \begin{align}\label{eq:def-rho}
 	\theta=\theta_{m,n-m}:= \frac{3n}{8m}\left(1-\sqrt{1-\frac{32}{9}\frac{m}{n}\left(\frac{n-m}{n}\right)} \right) \in \left[\frac{1}{2},\frac{2}{3}\right]\,.
 \end{align}
Then, assuming that $m/n\to b\in[0,1/2]$,
\begin{equation}
\lim_{n\to\infty}\sup_{\xi}	\E_\xi\left[\left\|\frac{\eta_{T(a)}}{\pi}-1\right\|^2_{2}\right] =\left(1+D(b)\right)e^{-a}\ ,\qquad a\in\R\ ,
\end{equation}
where 
\begin{equation}
	\label{eq:Db}	D(b):=\frac{3-4b-\sqrt{9-32b+32b^2}}{2\sqrt{9-32b+32b^2}}\in [0,\tfrac{\sqrt 6-2}{4}]\ ,\qquad b\in[0,1/2]\ .
\end{equation}
  \end{theorem} 

\medskip
We refer the reader to Section \ref{sec:comparison} below for further discussions around these results and for a comparison with  previously known facts.   The rest of the paper is organized as follows. In Section \ref{sec:2} we give formal definitions, recall and establish a few general facts that hold for ${\rm Avg}(G)$ on any graph $G$. In Section \ref{sec:hypercube} we prove Theorem \ref{th:hypercube}, in Section \ref{sec:complete-bipartite-l1} we prove Theorem \ref{th:K1}, and finally in Section \ref{sec:complete-bipartite-l2} we prove Theorem \ref{th:K2}.

\section{The Averaging process on arbitrary graphs: general facts}\label{sec:2}
The goal of this section is to collect a number of preliminary facts about  the Averaging process on an arbitrary (unweighted) graph. These general observations allow a first, not necessarily optimal, control on the mixing times.   Some of them can be found in the literature only in a discrete-time setting (e.g., \cite{movassagh_repeated2022}), others were formulated for more general reference measures $\pi$ (e.g., \cite{quattropani2021mixing}), or for weighted graphs (e.g., \cite{aldous_lecture_2012,quattropani2021mixing}); the entropy inequalities from \cite{bristiel_caputo_entropy_2021} were not originally related to the Averaging process. While not everything discussed in this section will be needed in the proof of our main results,  we believe that a unified presentation is beneficial and could be of use as a reference for future research.


\subsection{Averaging process vs.\ 
(lazy) random walk}
We start with a simple observation relating the expectation of ${\rm Avg}(G)$ with the transition probability of a random walk on the same graph. 

Given a finite connected graph $G=(V,E)$, with $|V|=n$, the Averaging process ${\rm Avg(G)}$ is the continuous-time Markov process $(\eta_t)_{t\ge 0}$ on the set $\Delta(V)$ of probability measures on $V$ with infinitesimal generator $\cL=\cL^{\scriptscriptstyle \rm Avg}$ given by
\begin{align}\label{eq:generator-avg}
	\cL f(\eta)=\sum_{xy\in E}\big(f(\eta^{xy})-f(\eta)\big)\ ,\qquad \eta \in \Delta(V)\ ,\ f : \Delta(V)\to \R\ ,
\end{align}
where $\eta^{xy}$ is defined as
\begin{align}\label{not:exy}
	\eta^{xy}(z):=\begin{dcases}
		\eta(z)&\text{if}\ z\neq x,y\\
		\frac{\eta(x)+\eta(y)}{2}&\text{otherwise}\ .
	\end{dcases}
\end{align}
When starting from a configuration $\xi \in \Delta(V)$, $\Prob_\xi$ and $\E_\xi$ denote the corresponding law and expectation, respectively.
The uniform measure $\pi\equiv 1/n$ is a fixed point of the dynamics, i.e., $\cL f(\pi)=0$ for all $f:\Delta(V)\to\R$. 

The process ${\rm RW}(G)$ is the continuous-time Markov chain $(X_t)_{t\ge0}$, with state space $V$, with generator
\begin{equation}\label{eq:generator-rw}
	\Lrw\psi(x)=\frac{1}{2}\sum_{y: \,xy\in E}\big(\psi(y)-\psi(x)\big)\ ,\qquad x\in V\ ,\  \psi:V\to \R\  .
\end{equation}
That is, $X_t$ is the simple random walk on $G$ with jump rate $1/2$. We will refer to it as \emph{lazy random walk} (or, simply, \emph{random walk}) on $G$, and call $\Prw_\xi$ the corresponding  law when starting from a distribution $\xi \in \Delta(V)$; when $\xi=\IND_x$, we simply write $\Prw_x$. Clearly, $\pi\equiv 1/n$ is the unique equilibrium distribution for ${\rm RW}(G)$. Further, we let $\lambda_{\rm gap}$ denote its \textit{spectral gap}, and $t_{\rm rel}$ its \textit{relaxation time}: $\lambda_{\rm gap}>0$ is the smallest non-zero eigenvalue of $-\Lrw$, and 
\begin{equation}\label{eq:def-trel}
	t_{\rm rel}=(\lambda_{\rm gap})^{-1}\ .
\end{equation}

The first observation is that, for all $x \in V$, $t\ge 0$ and $\xi \in \Delta(V)$,
\begin{equation}\label{eq:duality-bin1}
\E_\xi\left[\eta_t(x) \right]=\Prw_\xi(X_t=x)=:\pi_t^\xi(x)\ ,
\end{equation}
which equivalently rewrites, for some  collection of random variables  $Z_t^\xi=(Z_t^\xi(x))_{x\in V}$, as
\begin{equation}\label{eq:duality-bin1-v2}
  \eta_t=\pi^\xi_t+ Z_t^\xi\ ,\qquad \text{with}\ \eta_0=\xi\ ,\  
		\sum_{x\in V} Z_t^\xi(x)=0\ ,\ 
		\E_\xi[Z_t^\xi(x)]=0\ .
\end{equation}
In other words, the probability mass function at time $t$ is given by the deterministic random walk kernel $\pi^\xi_t$ plus a mean-zero noise term $Z_t^\xi$. The  identity in \eqref{eq:duality-bin1}, as emphasized in \cite{aldous_lecture_2012}, is an instance of \emph{duality}; see also \cite{liggett_interacting_2005-1}. An intuitive explanation of \eqref{eq:duality-bin1} is that, under repeated averages on $G$, any given \textquotedblleft infinitesimal chunk\textquotedblright\ of the initial mass performs a (lazy) random walk on $G$. 
As a first  immediate consequence of \eqref{eq:duality-bin1}--\eqref{eq:duality-bin1-v2}, one has the following general lower bound, stating that the presence of the noise term $Z_t^\xi$ cannot speed up  convergence to equilibrium.
\begin{proposition}[{\cite[Prop.\ 5.1]{quattropani2021mixing}}]\label{prop:general-lb}
	For all $p\in[1,\infty)$, $t\ge 0$, and  $\xi\in\Delta(V)$,
	\begin{equation}\label{eq:lower-bound-general}
		\E_\xi\left[\left\|\frac{\eta_t}{\pi}-1\right\|^p_p \right]\ge \bigg\|\frac{\pi_t^\xi}{\pi}-1\bigg\|^p_p\ .
	\end{equation}
\end{proposition} 
\begin{proof}
	By using the dual formulation of the $L^p$-distance, and calling $q$ the conjugate exponent of $p$, we get 
	\begin{align*}
		\E_\xi\left[\left\|\frac{\eta_t}{\pi}-1\right\|_p\right]&=	\E_\xi\left[\sup_{\|\psi\|_q=1}\sum_{x\in V}\left(\eta_t(x)-\pi(x)\right)\psi(x) \right]\\
		&\ge\sup_{\|\psi\|_q=1}\sum_{x\in V}\E_\xi\left[\left(\eta_t(x)-\pi(x)\right)\psi(x) \right]\\
 &=\sup_{\|\psi\|_q=1}\sum_{x\in V}\left(\pi^\xi_t(x)-\pi(x)\right)\psi(x) =\bigg\|\frac{\pi_t^\xi}{\pi}-1\bigg\|_p\ ,
	\end{align*}
where for the second identity we used \eqref{eq:duality-bin1}.
The  result follows by Jensen inequality.
\end{proof}
\begin{remark}[Sharpness of \eqref{eq:lower-bound-general}] As we will comment more in detail in Section \ref{sec:comparison} below, for certain families of graphs the lower bound in  \eqref{eq:lower-bound-general} is asymptotically optimal (e.g., for the hypercube), for others it is sub-optimal (e.g., for the $L^1$-distance on the complete graph, see \cite{chatterjee2020phase}).	
\end{remark}
The lower bound in Proposition \ref{prop:general-lb} is not the only  relation that one can derive between the $L^p$-distance of ${\rm Avg}(G)$ and that of ${\rm RW}(G)$. As first noticed in \cite{aldous_lecture_2012}, the infinitesimal mean contraction of the $L^2$-distance for ${\rm Avg}(G)$ can be expressed in terms of the Dirichlet form of ${\rm RW}(G)$, giving rise to a version of  Poincar\'e inequality for ${\rm Avg}(G)$. More precisely, recalling the definitions of infinitesimal generators $\cL$ and $\Lrw$ in \eqref{eq:generator-avg} and \eqref{eq:generator-rw}, respectively, 
and letting $\cErw$ denote the Dirichlet form for ${\rm RW}(G)$
\begin{align}
	\cErw(\psi)
	=\frac{1}{2n}\sum_{xy\in E}\big(\psi(x)-\psi(y)\big)^2\ ,\qquad \psi:V\to\R\ ,
\end{align}
a simple computation yields 
\begin{align}
	\frac{\dd}{\dd t}\,\E_\xi\left[\left\|\frac{\eta_t}{\pi}-1\right\|_2^2 \right]=\cL\, \E_\xi\left[\left\|\frac{\eta_t}{\pi}-1\right\|_2^2 \right]=- \E_\xi\left[\cErw\left(\frac{\eta_t}{\pi}\right)\right]\ .
\end{align}
As a consequence of this,  the variational formulation of the spectral gap,  Gr\"onwall inequality, and the Markov property of ${\rm Avg}(G)$ yield the following result.
\begin{proposition}[{\cite[Prop.\ 2]{aldous_lecture_2012}}]\label{prop:AL} Letting  $t_{\rm rel}$ be defined as in \eqref{eq:def-trel},
	\begin{equation}\label{eq:AL}
		\E_\xi\left[\left\|\frac{\eta_{t+s}}{\pi}-1 \right\|_2^2 \right]\le e^{-\frac{s}{t_{\rm rel}}}\, \E_\xi\left[ \left\|\frac{\eta_t}{\pi}-1 \right\|_2^2\right]\ ,\qquad \xi\in\Delta(V)\ ,\ s,t\ge 0\ .
	\end{equation}
\end{proposition}

\begin{remark}[Sharpness of \eqref{eq:AL}]
Compared to the more common Poincar\'e inequality for ${\rm RW}(G)$, the inequality in Proposition \ref{prop:AL} comes with a missing factor $2$ in the exponential contraction rate. In general, such an inequality cannot be improved, with the equality being attained when, for instance, $G$ is the complete graph (see \cite{chatterjee2020phase}).
\end{remark}
\subsection{Second moments and coupled random walks }
As already pointed out in \cite{aldous_lecture_2012}, also the second moments of $\eta_t$ admit a formulation in the spirit of  \eqref{eq:duality-bin1}, at the cost of replacing ${\rm RW}(G)$ by a 
suitable coupling of two random walks. 
 For a given graph $G=(V,E)$, we call
 ${\rm CRW}(G)$ the Markov process $(X_t,Y_t)_{t\ge 0}$ defined as follows. The state space is $V\times V$, interpreted as the set of positions of two labeled particles; as in ${\rm Avg}(G)$ we have independent rate 1 Poisson clocks at the edges of $G$; when an edge $xy$ rings, each particle on the endpoints of that edge is placed independently at $x$ or at $y$ with probability $1/2$. 
We note that as long as $X_t$ and $Y_t$ are at (graph) distance strictly larger than $1$, they evolve as two independent ${\rm RW}(G)$--processes, whereas while $X_t$ and $Y_t$ sit on the endpoints of the same edge (on the same vertex or on opposite vertices), they can experience synchronous jumps.

An argument analogous to that used for \eqref{eq:duality-bin1}, involving this time the evolution of two infinitesimal portions of masses of ${\rm Avg}(G)$, should convince the reader that the following identity holds: for all $x, y \in V$, $t\ge0$, and $\xi \in \Delta(V)$,
\begin{equation}\label{eq:duality-bin2}
	\E_\xi\left[\eta_t(x)\eta_t(y) \right]=\Purw_{\xi\otimes\xi}(X_t=x,Y_t=y)\ ,
\end{equation}
with $\Purw_{\xi\otimes \xi}$ denoting the law of ${\rm CRW}(G)$ when starting from $\xi\otimes \xi$ (two independent samples from $\xi$). By a generalization of the above two-particle process that goes under the name of {\em Binomial splitting process}, 
relations such as \eqref{eq:duality-bin2} extend to higher--order moments, see \cite{quattropani2021mixing}. 
	As the next proposition shows, \eqref{eq:duality-bin2} yields two convenient alternative formulations of the $L^2$-distance to equilibrium of ${\rm Avg}(G)$ in terms of the transition probabilities of ${\rm CRW}(G)$.
\begin{proposition}[\cite{aldous_lecture_2012,quattropani2021mixing}] \label{pr:expressions-L2-norm} For all $t\ge 0$ and $\xi\in \Delta(V)$,
	\begin{align}\label{eq:second-moments}
		\E_\xi\left[\left\|\frac{\eta_t}{\pi}-1  \right\|_2^2 \right]=n\,\Purw_{\xi\otimes \xi}(X_t=Y_t)-1
	\end{align}
and
\begin{equation}\label{eq:second-moments2}
	\E_\xi\left[\left\|\frac{\eta_t}{\pi}-1  \right\|_2^2 \right]	=\bigg\|\frac{\pi_t^\xi}{\pi}-1\bigg\|_2^2+\mathcal N^\xi_t\ ,
\end{equation}
where
\begin{equation}\label{eq:noise-term}
	\mathcal N^\xi_t:=\frac{n}{2}\int_{0}^{t}\sum_{xy\in E}\left(\pi^\xi_s(x)- \pi^\xi_s(y)\right)^2\Phi_{t-s}(x,y)\,\dd s\ ,
\end{equation}
and
\begin{equation}\label{eq:Phi-general}
		\Phi_{t}(x,y):=\frac{1}{2}\left(\Purw_{x,x}\left(X_t=Y_t \right)+\Purw_{y,y}\left(X_t=Y_t \right)-2\,\Purw_{x,y}\left(X_t=Y_t \right)\right)\in [0,1]\ .
\end{equation}
\end{proposition}
	The  identity in \eqref{eq:second-moments} follows at once  by expanding the $L^2$-norm and using \eqref{eq:duality-bin2}. The  identity in \eqref{eq:second-moments2} is slightly more involved, and can be obtained by an explicit comparison of the generator of the ${\rm CRW}(G)$--process and that of two independent ${\rm RW}(G)$--processes; we refer to  \cite[Prop. 5.5]{quattropani2021mixing} for the details.

\subsection{Worst--case initial condition}
In the literature on mixing times of Markov chains,  the distance to equilibrium at some time $t$ is most often analyzed by starting the chain from a \emph{worst--case initial condition}, namely one which maximizes the distance at time $t$. By a straightforward convexity argument, when the state space is finite,  worst--case initial conditions $\xi \in \Delta(V)$ are always Dirac distributions on some element of the chain's state space. An analogous argument, based on the \emph{linearity} of ${\rm Avg}(G)$, shows that the worst--case $L^p$ distance for ${\rm Avg}(G)$ is achieved at Dirac masses $\IND_x$, $x\in V$, for all $p\geq 1$.
A version of this fact already appeared in \cite{movassagh_repeated2022} for the case of the $L^1$-distance in the discrete-time model. The same  idea was also used in \cite{chatterjee2020phase} in the context of the complete graph, and in \cite{quattropani2021mixing} to estimate the $L^1$-distance of the Binomial splitting process  
on finite-dimensional graphs. The simple proof given here covers every graph and  all $L^p$-distances with $p\ge 1$.
\begin{proposition}\label{prop:worst-case}
	For any graph $G$, for all $t\ge 0$ and  $p\in [1,\infty)$,
	\begin{equation}
		\sup_{\xi \in \Delta(V)}\E_\xi\left[\left\|\frac{\eta_t}{\pi}-1 \right\|_{p}^p \right] = \sup_{x \in V}\E_{\IND_x}\left[\left\|\frac{\eta_t}{\pi}-1 \right\|_{p}^p \right]\ .
	\end{equation}
\end{proposition}
\begin{proof}
	The inequality \textquoteleft $\ge$\textquoteright\  is obvious. The proof of the inverse inequality follows by observing that, when the initial condition is $\xi\in \Delta(V)$, the probability mass function at time $t$ can be represented as 
	\begin{equation}
		\eta_t = \sum_{y\in V} \xi(y)\,  \eta^{(y)}_t\ ,\qquad  t \ge 0\ ,
	\end{equation}
	where, for each $y\in V$, $\eta^{(y)}_t$ is the averaging process on $G$ with initial condition $\IND_y$, and the evolutions $(\eta^{(y)}_t, \, t\geq 0)_{y\in V}$
are coupled by the use of the same edge update sequence. 
	Therefore, letting $\E$ denote the expectation with respect to this coupling,  
by Jensen inequality, 
	\begin{align}
		\E_\xi\left[\left\| \frac{\eta_t}{\pi}-1\right\|_{p}^p\right]
		&=\sum_{x\in V}\pi(x)\,\E\left[\bigg|\sum_{y\in V}\xi(y)\bigg(\frac{\eta^{(y)}_t(x)}{\pi(x)}-1\bigg) \bigg|^p \right]\\
		&\le \sum_{x\in V}\pi(x)  \sum_{y\in V}\xi(y)\, \E\left[\bigg|\frac{\eta^{(y)}_t(x)}{\pi(x)}-1 \bigg|^p \right]
		   = \sum_{y\in V}\xi(y)\,\E_{\IND_y}\left[\left\|\frac{\eta_t}{\pi}-1 \right\|_{p}^p \right]\ .	
	\end{align}
	This concludes the proof.
\end{proof}
\subsection{Relative entropy decay}
Relative entropy is  commonly used 
to measure the distance from stationarity, and it is natural to relate the mixing properties of a system to its exponential decay rate. 
In the usual Markov chain setup, this leads to the so-called Log-Sobolev and modified Log-Sobolev inequalities \cite{diaconis1996logarithmic,bobkov2006modified}. These inequalities control the global entropy functional in terms of a local Dirichlet form. In the setting of the averaging process, we are led to a new set of inequalities where the global entropy functional is controlled by local versions of the entropy itself rather than by a Dirichlet form.  Interestingly, these new inequalities were recently introduced and studied by \cite{bristiel_caputo_entropy_2021}  as a natural factorization statement for various random walk models, see also \cite{caputo_parisi_block_2021} for a spin system setting. However, the striking connection to the Averaging process that we point out here appears to be new. The latter can be formalized as follows. 

Let 
 \begin{align}
D(\eta||\pi):=\sum_{x\in V}\eta(x)\log\frac{\eta(x)}{\pi(x)}
 \ ,\qquad  \eta \in \Delta(V)\ ,
 \end{align}
denote  the \textit{relative entropy} (or \textit{Kullback--Leibler divergence}) of $\eta\in\Delta(V)$ with respect to $\pi\equiv 1/|V|$, and consider its expected value $\E_\xi \left[D(\eta_t||\pi)\right]$ under the Averaging process
 with initial condition $\xi\in\Delta(V)$.
 Recalling that $n=|V|$, the \emph{local entropy} of $\eta\in\Delta(V)$ at the edge $xy\in E$ is defined as 
 \begin{align}
{\rm ent}_{xy}(\eta):=
\frac{n}2\,\eta(x)\log\frac{\eta(x)}{\frac12(\eta(x)+\eta(y))}
+\frac{n}2\,\eta(y)\log\frac{\eta(y)}{\frac12(\eta(x)+\eta(y))}\ . \end{align}
\begin{proposition}\label{pr:entropy}
	The largest constant $\k\geq 0$ satisfying, for all $\xi\in\Delta(V)$ and $t\geq0$,
	 \begin{align}\label{eq:top}
\E_\xi\left[ D(\eta_t||\pi)\right]\leq e^{-\k\,t} D(\xi||\pi)\ ,
\end{align}
coincides with the largest constant $\k\geq 0$ satisfying, for all $\eta\in\Delta(V)$, 
	 \begin{align}\label{eq:top2}
\k\,	 D(\eta||\pi)\leq  \frac2n\sum_{xy\in E} {\rm ent}_{xy}(\eta)\ .
\end{align}
 \end{proposition}
 \begin{proof}
Recalling the notation \eqref{not:exy}, and observing that 
 \begin{equation}
\eta(x)\log\frac{\eta(x)}{\pi(x)}
+\eta(y)\log\frac{\eta(y)}{\pi(y)}
-\eta^{xy}(x)\log\frac{\eta^{xy}(x)}{\pi(x)}
-\eta^{xy}(y)\log\frac{\eta^{xy}(y)}{\pi(y)}=\frac{2}n\,{\rm ent}_{xy}(\eta)\ ,
 \end{equation}
we see that, as a result of an update at edge $xy$, the relative entropy of $\eta$ decreases from $D(\eta||\pi)$ to $D(\eta^{xy}||\pi)=D(\eta||\pi)-\frac{2}n\,{\rm ent}_{xy}(\eta)$. 
Therefore, 
 \begin{align}
	\frac{\dd}{\dd t}\E_\xi\left[ D(\eta_t||\pi)\right]
	=
	\cL\, 
\E_\xi\left[ D(\eta_t||\pi)\right]
=- \frac{2}n\,
\sum_{xy\in E} \E_\xi\left[{\rm ent}_{xy}(\eta_t)\right]\ .
\end{align}
By integrating over time, we see that \eqref{eq:top2} implies \eqref{eq:top}. On the other hand, if \eqref{eq:top} holds for all $t>0$, then subtracting $D(\xi||\pi)$ from both sides and taking the limit $t\to 0^+$ one obtains \eqref{eq:top2}. 
 \end{proof}
 Following \cite{bristiel_caputo_entropy_2021}, we call $\k=\k(G)$ the \emph{entropy constant} of the graph $G$. One can compute the value of $\k(G)$ for certain simple graphs. For instance, consider the trivial case where $G$ consists of a single edge $xy$. In this case $n=2$ and $D(\eta||\pi)={\rm ent}_{xy}(\eta)$ so that $\k=1$. Therefore, using the product structure one easily checks that the entropy constant of the hypercube $\{0,1\}^d$ is again $\k=1$ for all $d\geq1$. 
We refer to \cite{bristiel_caputo_entropy_2021} for other examples including random walks on hypergraphs and processes with more than one particle such as the Binomial splitting process. 

 Proposition \ref{pr:entropy} is the entropic analogue of the $L^2$-statement in Proposition \ref{prop:AL}.
 A linearization argument shows that the constant $\k=\k(G)$ in \eqref{eq:top} satisfies $\k(G)\leq 1/t_{\rm rel}(G)$ for any graph $G$.  
 Moreover, one can show (see \cite[Lem.\  2.1]{bristiel_caputo_entropy_2021}) that $\k(G)\geq 2\log(2)\b(G)$, for any $G$, where $\b(G)$ denotes the Log-Sobolev constant of $G$ defined as the largest constant $\b\geq 0$ satisfying, for all $\eta\in\Delta(V)$,
 \begin{align}\label{eq:logsob}
\b\,D(\eta||\pi)\leq 	\cErw(\sqrt{\eta/\pi})\ .
\end{align}
In particular, this implies that the entropic bound $\E_\xi\left[ D(\eta_t||\pi)\right]\leq e^{-\b\,t}D(\xi||\pi)$ proposed in   
\cite[Prop.\ 7]{aldous_lecture_2012} can be improved by a factor $2\log(2)$ in the exponential decay rate. 
Further relations of interest include upper bounds on $\kappa(G)$ in terms of both $\beta(G)$ and the Modified Log-Sobolev constant of $G$, see \cite{bristiel_caputo_entropy_2021} for more details.

Finally, Proposition \ref{pr:entropy} and Pinsker's inequality
\begin{align}\label{eq:pinsker}
	\left\|\frac{\eta}{\pi}-1\right\|_1\le \sqrt{2D(\eta||\pi)}\ ,\qquad \eta\in\Delta(V)\ ,
\end{align}
provide the following upper bound on $L^1$-distance in terms of the entropy constant $\k$.
\begin{corollary}\label{cor:entropy}
For all $\xi\in\Delta(V)$, and for all $t\geq0$
	 \begin{align}\label{eq:top3}
\E_\xi\left[\left\|\frac{\eta_t}{\pi}-1  \right\|_1 \right]
\leq e^{-\k\,t/2} \sqrt {2\log n},
\end{align}
where $\k=\k(G)$ denotes the entropy constant of $G$.
 \end{corollary}
 \begin{proof}
 From \eqref{eq:pinsker} and Cauchy-Schwarz inequality, one has 
	 \begin{align}\label{eq:top4}
\E_\xi\left[\left\|\frac{\eta_t}{\pi}-1  \right\|_1 \right]\leq\sqrt{2\,\E_\xi\,D(\eta_t||\pi)} 
\leq e^{-\k\,t/2} \sqrt{2D(\xi||\pi)}\ .
\end{align}
By convexity, $D(\xi||\pi)$ is maximized when $\xi$ is a Dirac mass at a vertex;  	therefore, since $\pi\equiv 1/n$, $D(\xi||\pi)\leq \log n$.
 \end{proof}

\subsection{A lower bound using relative entropy}
Concerning lower bounds, the following statement is the continuous-time version of a useful estimate established in the recent work \cite{movassagh_repeated2022}.
\begin{proposition}[{\cite[Th.\ 1]{movassagh_repeated2022}}]\label{prop:MSW}
Let
\begin{equation}\label{eq:lb-time-MSW}
 T_\varepsilon:=(1-\varepsilon)\frac{\log_2 n}{\langle {\rm deg}\rangle}\, ,\qquad \varepsilon \in (0,1)\,,
\end{equation}
where ${\rm deg}(x):=$ degree of $x\in V$ in the graph $G$, and $\langle {\rm deg}\rangle:= \frac{1}{n}\sum_x {\rm deg}(x)$.
Then, 
	\begin{equation}\label{eq:lb-MSW}
		\liminf_{n\to\infty}\sup_{\xi\in\Delta(V)} \E_{\xi}\left[\left\|\frac{\eta_{T_\varepsilon}}{\pi} -1\right\|_1 \right]\geq\varepsilon\ .
	\end{equation}
\end{proposition}
\begin{proof}
	As in the proof of Proposition \ref{pr:entropy}, an update at edge $xy$ at time $t$ causes the relative entropy of $\eta_t$ to decrease from $D(\eta_t||\pi)$ to $D(\eta^{xy}_t||\pi)=D(\eta_t||\pi)-\frac{2}n\,{\rm ent}_{xy}(\eta_t)$. Since  $\frac{2}n\,{\rm ent}_{xy}(\eta_t)\leq(\eta_t(x)+\eta_t(y)) \log 2$,  integrating and using \eqref{eq:duality-bin1}, one has
	\begin{align}\label{eq:MSW1}
		\begin{aligned}
	\E_\xi\left[D(\eta_t||\pi)\right]&\ge D(\xi||\pi)-\log 2\int_0^t \sum_{xy\in  E} \E_\xi\left[\eta_s(x)+\eta_s(y)\right]\dd s\\
	&= D(\xi||\pi)-\log 2\int_0^t \sum_{x\in V}\pi^\xi_s(x)\, {\rm deg}(x)\, \dd s\, .
	\end{aligned}
	\end{align}
Since $\langle {\rm deg}\rangle-{\rm deg}$ has mean zero with respect to $\pi$, there exists 
	 $\psi:V\to \R$ solving	
	 \begin{equation}\label{eq:solution-psi}
	\Lrw\psi (x)=\log 2\:(\langle {\rm deg} \rangle -{\rm deg}(x))\ ,\qquad x \in V\ .
	\end{equation}
	Moreover, $\psi$ is defined up to constants, so that we are allowed to select a translation of $\psi$ satisfying
	\begin{equation}\label{eq:conditions-translate}
	\psi\ge 0\ ,\qquad \text{and}\qquad \psi(y)=0\quad \text{for some}\ y \in V\ .
	\end{equation}
Note that if $G$ is regular (i.e., $ {\rm deg}\equiv \text{const.}$), then we can take $\psi\equiv0$.
	Introducing  the function 
	\begin{equation}
		\label{eq:augmented-entropy}
		F(\eta):=D(\eta||\pi)-\sum_{x\in V}\eta(x)\, \psi(x)\ , \qquad \eta\in\Delta(V)\ ,
		\end{equation}
	the two conditions in \eqref{eq:conditions-translate} ensure, respectively,
	\begin{align}\label{eq:F<D}
		F(\eta_t)\le D(\eta_t||\pi)\ ,\qquad \Prob_{\IND_y}\text{-a.s.}\ ,
	\end{align}
and
\begin{align}\label{eq:F0} F(\IND_y)=D(\IND_y||\pi)=\log n\ .
\end{align}
By combining  \eqref{eq:augmented-entropy},  \eqref{eq:duality-bin1}, \eqref{eq:F0},   \eqref{eq:MSW1}, and writing $\pi_t^y=\pi_t^{\IND_y}$,
\begin{align}
	\E_{\IND_y}\left[F(\eta_t)\right]&= \E_{\IND_y}\left[D(\eta_t||\pi)\right]-\sum_{x\in V}\pi^y_t(x)\, \psi(x)\\
	&\ge \log n- \int_0^t\sum_{x\in V}\left( \log 2\,\pi^y_s(x)\,  {\rm deg}(x)+\big(\Lrw\pi^y_s\big)(x)\, \psi(x) \right)\dd s\\
	&= \log n- \int_0^t \sum_{x\in V} \pi^y_s(x)\left( \log 2\, {\rm deg}(x) +\Lrw \psi(x)\right)\dd s
	\, 	
\end{align}
where for the last identity we used that $\Lrw$ is symmetric.   
By \eqref{eq:solution-psi}, the expression inside the time integral above equals $\log 2\, \langle {\rm deg}\rangle$, thus, using \eqref{eq:F<D},
\begin{equation}
\E_{\IND_y}	\left[D(\eta_t||\pi)\right]\geq \E_{\IND_y}\left[F(\eta_t)\right]\ge \log n- \langle {\rm deg}\rangle\, t\log 2\,,\qquad t\geq0\,,
\end{equation}
and choosing $t=T_\varepsilon$ as in \eqref{eq:lb-time-MSW} one has $\E_{\IND_y}	D(\eta_t||\pi)\geq\varepsilon \log n = \varepsilon D(\IND_y||\pi)$. 
To obtain the desired lower bound on the $L^1$-norm one uses the fact that, for $\pi$ uniform,   Fannes--Audenaert inequality (see, e.g., \cite{audenaert_sharp_2007}) 
\begin{align}\label{eq:fannes-audenaert}
	 \left\|\frac{\eta}{\pi}-1 \right\|_1\ge \frac{D(\eta||\pi)}{\log n}-\frac{1}{e \log n}\ ,\qquad \eta \in \Delta(V)\ ,
\end{align}
holds true. This concludes the proof.
\end{proof}
\begin{remark}
The above proof shows that  $\E_{\IND_y}	D(\eta_{T_\varepsilon}||\pi)\geq\varepsilon D(\IND_y||\pi)$. It follows from Proposition \ref{pr:entropy} that $\k(G)T_\varepsilon \leq \log(1/\varepsilon)$, and optimizing over $\varepsilon\in(0,1)$, one obtains the general bound 
\begin{align}\label{eq:upbokappa}
\k(G)\leq \frac{ \langle {\rm deg}\rangle}{\log_2 n}\,.
\end{align}
Remarkably, if $G=K_n$ is the complete graph, the estimate \eqref{eq:upbokappa} captures the exact value $\k(K_n)=(n-1)/\log_2(n)$ (see \cite[Remark 1.2]{bristiel_caputo_entropy_2021} and note that a factor $2$ has to be removed there because our random walk has jump rate $1/2$). Moreover, the bound \eqref{eq:upbokappa} is saturated also when $G$ is the  hypercube, in  which case one has $\kappa(G)=1$.
\end{remark}
\begin{remark}
 For certain families of graphs such as the complete graphs, the  lower bound in \eqref{eq:lb-MSW}  predicts very sharply the order of the $L^1$-mixing time for ${\rm Avg}(G)$, see the  discussion in Section \ref{sec:comparison}. However, due to the $\varepsilon$-dependence
in \eqref{eq:lb-time-MSW}, this lower bound is not as sharp as needed to establish a cutoff phenomenon. 
\end{remark}

\subsection{Comparison of mixing times and summary of the main results}\label{sec:comparison}
Following an established convention, we let the $L^p$-mixing time of ${\rm RW}(G)$ be defined as the first time $t$ such that the $L^p$-distance $\|\pi_t^\xi/\pi-1\|_p$ is at most $1/2$, regardless of the initial condition $\xi$. The  $L^p$-mixing time of ${\rm Avg}(G)$ is defined in the same way by using the averaged quantity \eqref{eq:tfun} instead of the $L^p$-distance. In Table \ref{table:1} and Table \ref{table:2} below we summarize all known estimates about such mixing times, including the results in the present paper. 
We display, in particular:
\begin{enumerate}

  \item lattice approximations of $d$-dimensional boxes $G=[0,1]^d\cap \frac{\mathbb Z^d}{n}$ (this extends to  all finite-dimensional graphs, see \cite{quattropani2021mixing}); 
  
  \item 
   hypercube $G=\{0,1\}^d$ (Theorem \ref{th:hypercube}); 
   
   \item
    complete graph $G=K_n$ (see \cite{chatterjee2020phase}); 
    
    \item
     complete bipartite graphs $G=K_{m,n-m}$ (Theorem \ref{th:K1} and Theorem \ref{th:K2}).
\end{enumerate}

The purpose of these tables is that of comparing mixing times; hence,  we express them all  in terms of   the relaxation time $t_{\rm rel}$ of ${\rm RW}(G)$ defined in \eqref{eq:def-trel}. As a consequence,  such expressions will depend on the specific choice of the time normalization solely through the exact value of $t_{\rm rel}$.

\renewcommand\arraystretch{1.7}
\begin{table}[h]
	\begin{tabular}{c|c|c|} 
		& $L^1$, ${\rm RW}(G)$ & $L^1$, ${\rm Avg}(G)$   \\\hline
	 $[0,1]^d\cap \frac{\mathbb Z^d}{n}$ &$\Theta(t_{\rm rel})$   &$\Theta(t_{\rm rel})$    \\\hline
		 $\{0,1\}^d$	   &$\frac{t_{\rm rel}}{2}\left(\log d +\Theta(1)\right)$ & $\frac{t_{\rm rel}}{2}\left(\log d +\Theta(1)\right)$
		\\ \hline
		$K_n$	  & $\Theta(t_{\rm rel})$ & $\frac{1}{2\log 2}\,t_{\rm rel}\left(\log n+\Theta(\sqrt{\log n})\right)	$ 
		\\
		\hline
		$K_{m,n-m}$ 	 &$\Theta(t_{\rm rel})$  &  $\frac{n}{n-m}\frac{1}{4\log 2}\,t_{\rm rel}\left(\log n+\Theta(\sqrt{\log n})\right)$\\\hline
	\end{tabular}
	\vspace{2ex}
	\caption{$L^1$-mixing times of ${\rm RW}(G)$ and ${\rm Avg}(G)$.}
	\label{table:1}
	\begin{tabular}{c|c|c|} 
		 &$L^2$, ${\rm RW}(G)$ & $L^2$, ${\rm Avg}(G)$  \\\hline
		  $[0,1]^d\cap \frac{\mathbb Z^d}{n}$   & $\Theta(t_{\rm rel})$ &$\Theta(t_{\rm rel})$  \\\hline
		  $\{0,1\}^d$	   &$\frac{t_{\rm rel}}{2}\left(\log d +\Theta(1)\right)$ & $\frac{t_{\rm rel}}{2}\left(\log d +\Theta(1)\right)$
		  \\ \hline
		  	$K_n$	 & $\frac{t_{\rm rel}}{2}\left(\log n+\Theta(1)\right)$ &$t_{\rm rel}\left(\log n+\Theta(1)\right)$
		  \\
		  \hline
		  	$K_{m,n-m}$ 	 & $\frac{t_{\rm rel}}{2}\left(\log n +\Theta(1)\right)$& $C(\frac{m}{n})\,t_{\rm rel}\left(\log n +\Theta(1)\right)$\\\hline
	\end{tabular}
	\vspace{2ex}
	\caption{$L^2$-mixing times of ${\rm RW}(G)$ and ${\rm Avg}(G)$; $C(\cdot)$ is defined in \eqref{eq:C} below.}
		\label{table:2}
\end{table}
In Table \ref{table:2}, the function $b\in (0,1/2]\mapsto C(b)\in (3/4,1]$ is given by (cf.\ Rmk.\ \ref{rem:cutoff-trel} below)
\begin{equation}\label{eq:C}
	C(b):= \frac{4}{3}\frac{b}{1-\sqrt{1-\frac{32}{9}b\left(1-b\right)}}\ ,\qquad\text{and}\qquad \lim_{b\downarrow 0}C(b)=\frac{3}{4}\ ,\quad C(\tfrac{1}{2})=1\ .
\end{equation}

We conclude with a discussion of the main results of this paper and some comments on the strategy of proof together with comments on the applications of the general bounds presented in this section. 

\subsubsection{Hypercube}\label{sec:summary-hypercube-fdg} Here $G=\{0,1\}^d$, $n=2^d$, and with our time normalization one has $t_{\rm rel}=1$. Thus, Theorem \ref{th:hypercube} implies that  mixing of ${\rm Avg}(G)$  on the hypercube occurs with cutoff at the same times $T=\frac{t_{\rm rel}}{2}\left(\log d+\Theta(1)\right)$ of that of ${\rm RW}(G)$; we refer to \cite[Example, p.\ 987]{lubetzky_sly_2014} for further details on mixing for ${\rm RW}(G)$ on the hypercube. 
Therefore, by combining this with  Proposition \ref{prop:general-lb}, we readily obtain an explicit lower bound for $L^1$- and $L^2$-mixing of ${\rm Avg}(G)$. By Jensen inequality, Theorem \ref{th:hypercube} boils down, then, to establishing a matching upper bound for the $L^2$-mixing time of ${\rm Avg}(G)$.

 For this purpose,  let us observe that 
the general approach via $L^2$-contraction from Proposition \ref{prop:AL}  would give only $T\le t_{\rm rel}\, \Theta(d)$.
On the other hand, the entropy contraction from  
Proposition \ref{pr:entropy} yields $T\le t_{\rm rel}\left(\log d+\Theta(1)\right)$ for the $L^1$-distance, that is an upper bound off by a factor $2$. Indeed, this latter bound follows from Corollary \ref{cor:entropy} and the already mentioned fact that $\k(G)=1$ when $G$ is the hypercube. Instead, to prove the upper bound in Theorem \ref{th:hypercube}, we provide a tight estimate for ${\rm CRW}(G)$ which controls the $L^2$-distance via the first decomposition in Proposition \ref{pr:expressions-L2-norm}; see Proposition \ref{prop:l2cutoff-S} below.

Furthermore, we observe that  the lower bound based on relative entropy in Proposition \ref{prop:MSW} does not provide the correct order of magnitude for the $L^1$-mixing time in the hypercube: since $\langle {\rm deg}\rangle= d$, this yields $T\ge (1-\varepsilon)\, t_{\rm rel}$ for every $\varepsilon \in (0,1)$.
Finally, we remark that the hypercube example seems to disprove Conjecture 1 in \cite{movassagh_repeated2022}. Indeed, mixing for ${\rm Avg}(G)$ occurs at times $\Theta(t_{\rm rel}\log \log n)$, while the authors in \cite{movassagh_repeated2022} conjectured that  mixing  must occur at times either of order $\Theta(t_{\rm rel})$ or $\Theta(t_{\rm rel}\log n)$.

\subsubsection{Complete bipartite graphs}\label{suse:strategy-bipartite}
For complete bipartite graphs $K_{m,n-m}$, both the $L^1$- and $L^2$-mixing times
for ${\rm Avg}(G)$
are strictly larger than their analogues for ${\rm RW}(G)$,
(cf.\ Tables \ref{table:1} and Table \ref{table:2}) and therefore the lower bound of Proposition \ref{prop:general-lb} turns out to be too poor. 
Moreover, the leading-order constants depend also on $\lim_{n\to \infty}\frac{m}{n}\in [0,1/2]$: when $m=\frac{n}{2}$ (regular case), one recovers the complete-graph values of the mixing times for both $L^1$- and $L^2$-mixing, asymptotically saturating Aldous-Lanoue's inequality in \eqref{eq:AL}; when $m=1$ (star graph), $L^1$-mixing occurs strictly in between $L^1$- and $L^2$-mixing for ${\rm RW}(G)$, while $L^2$-mixing for ${\rm Avg}(G)$  sits strictly in between $L^2$-mixing for ${\rm RW}(G)$ and the upper bound prescribed by \eqref{eq:AL}. Analogous observations hold for all other intermediate regimes of $m$.

For the  $L^1$ and $L^2$ mixing we use two different strategies of proof. For what concerns $L^1$, we extend to complete bipartite graphs  $G=K_{m,n-m}$ the probabilistic arguments developed in \cite{chatterjee2020phase}  for the complete graph $K_n$. 
As for $L^2$, due to the symmetries of the underlying graph, we exploit a representation of the $L^2$-distance in terms of a simpler $5$-state Markov chain, and analyze  its asymptotic behavior in detail.

We remark that, as for the case of the complete graph considered in \cite{movassagh_repeated2022}, the entropic lower bound of Proposition \ref{prop:MSW} predicts the correct order of the $L^1$-mixing time also for complete bipartite graphs. On the contrary, the upper bound based on relative entropy contraction of Proposition \ref{pr:entropy} yields an estimates that is off by a  $\log\log n$ factor in this case.

\section{Mixing on the hypercube}
\label{sec:hypercube}
This section is devoted to the proof of Theorem \ref{th:hypercube}. We let $G=(V,E)$ be the $d$-hypercube $V=\{0,1\}^d$, whose edge set consists of unordered pairs $\{x,x^i\}$, with $x=(x_1,\ldots, x_d)\in \{0,1\}^d$ and 
\begin{align}
	x^i_j=(x^i)_j:=\begin{dcases}
		x_j &\text{if}\ j\neq i\\
		1-x_i &\text{otherwise}\  .
	\end{dcases}
\end{align}
By Proposition \ref{prop:worst-case} and transitivity of the underlying graph, \emph{all} Dirac masses are  worst-case initial conditions for ${\rm Avg}(G)$. Hence, throughout this section,  we fix an arbitrary $x_0\in V$, and set $\xi=\IND_{x_0}$ as an initial condition. 

As already discussed in Section \ref{sec:summary-hypercube-fdg}, the claim in Theorem \ref{th:hypercube} reduces to estimating from above the $L^2$-mixing time of  ${\rm Avg}(G)$.
 In view of \eqref{eq:second-moments}, we have, for all $t\ge 0$,
\begin{equation}\label{eq:l2-hypercube}
 	\E_{\IND_{x_0}}\left[\left\|\frac{\eta_t}{\pi}-1  \right\|_2^2 \right]	=2^d\:\Purw_{x_0,x_0}(X_t=Y_t)-1\ .
 \end{equation}
	In order to estimate the probability on the right-hand side of \eqref{eq:l2-hypercube}, we analyze the  Hamming distance (cf.\ \eqref{eq:hamming-distance} below) between these two coupled walks. It is well-known that, in absence of coupling effects, i.e., for two independent copies of ${\rm RW}(G)$ on $G=\{0,1\}^d$, the process induced by this projection corresponds  to the  Markov chain on $\{0,1,\ldots, d\}$ known as Ehrenfest urn (see, e.g., \cite[Sect.\ 2.3]{levin2017markov}). The Hamming distance between the two coordinates of ${\rm CRW}(G)$ still yields a birth-and-death chain, but this time with a \textquotedblleft defect\textquotedblright\ due to the interaction.
	  For the purpose of proving Theorem \ref{th:hypercube}, we carry out an analysis of the $L^2$-mixing time of such a perturbation of the  Ehrenfest urn. 
Mixing with respect to $L^1$- and separation distances for perturbations of birth-and-death chains has been studied 
in \cite{diaconis_saloff-coste_separation_2006} and \cite{chen_saloff-coste_computing_2015}, see, e.g., \cite[Th.\ 4.10]{chen_saloff-coste_mixing_bd_chains_2013}. However, for the $L^2$-distance, the situation seems to be much less generally understood. For this we combine general results on birth-and-death chains with the recent characterization \cite{hermon_peres_characterization_2018} of $L^2$-mixing times in the reversible context.

\subsection{A perturbation of the Ehrenfest urn}\label{sec:ehrenfest}  We start with some definitions. Consider the two birth-and-death chains obtained from  both  ${\rm RW}(G)\otimes{\rm RW}(G)$ (i.e., two independent copies of ${\rm RW}(G)$) and ${\rm CRW}(G)$ via  the projection
\begin{equation}\label{eq:hamming-distance}
	(x,y)\in \{0,1\}^d\times \{0,1\}^d\longmapsto |x-y|:= \sum_{i=1}^d |x_i-y_i| \in \{0,1,\ldots, d\}\ .
\end{equation}
Let $P_t(k,\ell)$ and $S_t(k,\ell)$, $k,\ell \in \{0,1,\ldots, d\}$,
denote the   transition kernels of these two Markov chains,  and let $\cL^P$ and $\cL^S$ denote their infinitesimal generators, respectively. Writing $p^{P,S}_k$, resp.\  $q^{P,S}_k$, for the associated birth, resp.\ death, rates  when in state $k\in \{0,1,\ldots, d\}$, we have
\begin{equation}\label{eq:rates-P}
p^P_k=d-k\ ,\qquad	 q^P_k= k\ ,
\end{equation}
and
\begin{equation}\label{eq:rates-S}
p^S_k= \IND_{k\neq 0}\, p^P_k+\IND_{k=0}\, \frac{d}{2}\, ,\qquad q^S_k= \IND_{k\neq 1}\, q^P_k + 1_{k=1}\, \frac{1}{2} \ ,
\end{equation}
while the Binomial distribution $\nu$ given by
\begin{equation}\label{eq:nu-bin}\nu(k):= \binom{d}{k}\, 2^{-d}\ ,\qquad k\in \{0,1,\ldots, d\}\ ,
	\end{equation}
is reversible for both chains.
Finally, we let $P_t$ and $S_t$ represent the Markov semigroups of the (continuous--time) Ehrenfest urn and its perturbation, respectively. 

With these definitions,  the right-hand side of \eqref{eq:l2-hypercube} further reads as
\begin{align}
	2^d S_t(0,0)-1
	& = \left\| S_t\left(\frac{\IND_0}{\nu}\right)-1\right\|_{L^\infty(\nu)}
	= \left\| S_{t/2}\left(\frac{\IND_0}{\nu} \right)-1 \right\|_{L^2(\nu)}^2 \,.
\end{align}
The desired upper bound on the $L^2$-mixing of ${\rm Avg}(G)$ at times $\frac{1}{2}\log d+\Theta(1)$ then follows by \eqref{eq:l2-hypercube} as soon as the $S$-chain is well-mixed in $L^2$ at times $\frac{1}{4}\log d+ \Theta(1)$. This last claim is established in the next proposition, which is the main technical result of this section.
\begin{proposition}\label{prop:l2cutoff-S}
	For every $\varepsilon>0$, there exists $c=c_\varepsilon>0$ such that, for all $d$ sufficiently large, 
	\begin{equation}\label{eq:prop-l2cutoff-S}
		\left\| S_t\left(\frac{\IND_0}{\nu} \right)-1 \right\|_{L^2(\nu)}^2\le \varepsilon
	\end{equation}
holds for all $t>\tfrac14 \log d +c$. 
\end{proposition}
The proof of this result is postponed to Section \ref{sec:proof-prop-3.1} below.

 \begin{remark} By well-known monotonicity properties of birth-and-death chains (see, e.g., \cite[Lem.\ 4.1]{ding_lubetzky_peres_total_variation_2010}), we have the following comparison inequalities:
\begin{equation}\label{eq:sandwich}
P_t(0,0)\le  S_t(0,0)\le P_{t/2}(0,0)\ ,\qquad t \ge  0\ .
\end{equation}
Since the $P$-chain mixes abruptly at times $t\approx \frac{1}{4}\log d$, the bounds in \eqref{eq:sandwich} are insufficient to establish  cutoff for the $S$-chain. 
In this respect, Proposition \ref{prop:l2cutoff-S} tells us that the first inequality in \eqref{eq:sandwich} is sharp for times $t\approx \tfrac14 \log d$. Therefore, the second inequality --- which, roughly speaking, accounts for the early delay of the $S$-chain when exiting state $0$ ---  needs to be improved for such times.  \end{remark}


\subsection{Proof of Proposition \ref{prop:l2cutoff-S}}\label{sec:proof-prop-3.1}
The idea of the proof Proposition \ref{prop:l2cutoff-S} is based on the probabilistic approach to $L^2$-mixing recently introduced in \cite{hermon_peres_characterization_2018}, which links $L^2$-mixing times to estimates of suitable hitting times. Building on such a machinery, the following lemma represents the main technical step to derive the claim in Proposition \ref{prop:l2cutoff-S}. Its proof combines two ingredients: the knowledge of the distribution of hitting times for birth-and-death chains (\cite{brown_shao_1987}), and a discrete weighted Hardy's inequality (\cite{miclo_example_1999}). As a result, we obtain a comparison between the $P$-chain and the $S$-chain that is much finer than the one implied by \eqref{eq:sandwich}. This, together with the results in \cite{hermon_peres_characterization_2018}, provides a tight comparison between the mixing behavior of the two chains.

In what follows, $\Prob^{P,S}_k$ denotes the law of the $P$- or $S$-chain defined in Section \ref{sec:ehrenfest} when starting from state $k\in \{0,1,\ldots, d\}$, while $\nu([0,M)):= \sum_{k=0}^{M-1}\nu(k)$ (cf.\ \eqref{eq:nu-bin}). 
\begin{lemma}\label{lemma:final}
Define $\tau_M$ to be the first hitting time of the state $M\in\{0,1,\dots, d\}$. Then, for every $\sigma>0$,  there exists $b=b(\sigma)>0$ satisfying, for all   $d$ large enough,
\begin{equation}
	\Prob^S_0\left(\tau_M>t+b\right)\le \left(\nu([0,M))\right)^\sigma+ \Prob^P_0\left(\tau_M>t\right) \ ,\qquad t>0\ , \ M\in\{1,\dots,d/2\}\ .
\end{equation}
\end{lemma}
\begin{proof}
As shown in \cite{brown_shao_1987} (see also, e.g., \cite{chen_saloff-coste_mixing_bd_chains_2013}), for any birth-and-death chain on $\{0,1,\ldots, d\}$ with rate matrix $\cL$ started in $k=0$, the hitting time $\tau_M$ of the target $M \in \{1,\ldots, d\}$ is distributed as the sum of $M$ independent exponential random variables with  rates given by   $(\lambda_{M,i})_{i=0}^{M-1}$, these being the distinct eigenvalues of $-\cL|_M$, the generator of the sub-Markovian chain $\cL$ killed upon exiting $\{0,1,\ldots, M-1\}$. 

Now, while the eigenvalues and eigenfunctions of the Ehrenfest urn $\cL^P$ defined in Section \ref{sec:ehrenfest} are well-known (see, e.g., \cite{karlin_mcgregor_ehrenfest_1965}), neither those of  $\cL^P|_M$ nor of $\cL^S|_M$ are explicit. Nevertheless, a comparison of eigenvalues is available. Letting $U:= \diag(\sqrt \nu)$ and	$\tilde \cL^{P,S}:= U\cL^{P,S} U^{-1}$, we have	\begin{align}
	-\tilde \cL^S=-\tilde \cL^P+\cR\ ,	
\end{align}
with $\cR=(\cR_{ij})_{i,j=0}^d$ being the following rank-one matrix: $\cR_{ij}=0$ if $i+j>1$,	 and
\begin{equation}
	\cR_{00}=-\frac{d}{2}\ ,\qquad \cR_{01}=\cR_{10}=\sqrt{\frac{d}{2}}\ ,\qquad \cR_{11}=-\frac{1}{2}\ .
\end{equation}
By removing, for every $M\in \{1,\ldots, d\}$, the last $d-M+1$  rows and columns of such matrices, we recover analogous relations for the sub-Markovian chains killed upon exiting  $\{0,1,\ldots, M-1\}$; in particular, the symmetric matrix $-\tilde \cL^S|_M$ is a rank-one perturbation of $-\tilde \cL^P|_M$. Therefore, Weyl theorem  (see, e.g., \cite[Th.\ 4.3.1]{horn_matrix_2012})  yields \begin{equation}\label{eq:eigenvalues1}
(\lambda_{M,i-1}^P\vee \lambda_{M,i-1}^S)\le	 \lambda_{M,i}^S\le \lambda_{M,i}^P \ ,\qquad i=0,\ldots, M-1\ .
	\end{equation}
Moreover, the variational characterization of the spectral bounds $\lambda_{M,0}^P$ and $\lambda_{M,0}^S$ and the definitions of rates in \eqref{eq:rates-P}--\eqref{eq:rates-S} ensure that
\begin{equation}\label{eq:eigenvalues2}
\frac{1}{2}\,\lambda_{M,0}^P\le 	\lambda_{M,0}^S \ .
\end{equation} 

Now, let $Z^{P,S}_{M,i}, \tilde Z^{P,S}_{M,i}\sim {\rm Exp}(\lambda^{P,S}_{M,i})$ be all mutually independent exponential r.v.'s, with $\Prob$ denoting their joint law; then,  for all $s>0$ and $M\in \{1,\ldots, d\}$, \eqref{eq:eigenvalues1}--\eqref{eq:eigenvalues2} yield
\begin{align}
	 \Prob\left(\sum_{i=0}^{M-1}Z^S_{M,i}>s\right)
		\le \Prob\left(Z^S_{M,0}+\sum_{i=1}^{M-1} Z^P_{M,i-1}>s\right)\le \Prob\left(2\,\tilde Z^P_{M,0}+\sum_{i=0}^{M-1}Z^P_{M,i}>s\right)\ .
\end{align}
 Setting $s=t+b>0$ in the above inequality,	by independence, we further get
 \begin{align}
 	\Prob\left(\sum_{i=0}^{M-1}Z^S_{M,i}>t+b\right)\le \Prob\left( Z^P_{M,0}>b/2\right)+\Prob\left(\sum_{i=0}^{M-1} Z^P_{M,i}>t\right)\ .
 \end{align}
In view of $Z^P_{M,0}\sim {\rm Exp}(\lambda^P_{M,0})$ and the aforementioned explicit distribution of  $\tau_M$ under $\Prob_0^{P,S}$,  the above inequality reads as 
\begin{align}\label{eq:final-lemma}
	\Prob^S_0\left(\tau_M>t+b\right)\le \exp\left(-\frac{b}{2}\,\lambda^P_{M,0}\right) + \Prob^P_0\left(\tau_M>t\right)\ .
\end{align}
We estimate the first term on the right-hand side above in Appendix \ref{app:est-hardy} below, by proving
\begin{align}\label{eq:eigenvaluessss2}
	\limsup_{d\to \infty}\sup_{1\le M\le d/2}\left(\lambda^P_{M,0}\right)^{-1} \log\left(\frac{1}{ \nu([0,M))}\right)<\infty\ ,
\end{align}
which, combined with \eqref{eq:final-lemma}, yields the assertion of the lemma.
\end{proof}

We are now in a good position to prove Proposition \ref{prop:l2cutoff-S}.
\begin{proof}[Proof of Proposition \ref{prop:l2cutoff-S}]
	We start by recalling some of the results in \cite{hermon_peres_characterization_2018} (namely, Prop.\ 3.8, and Thm.s\ 1.2 \& 5.1), suitably adapted to the context of a generic birth-and-death chain on $\{0,1,\ldots, d\}$ with transition matrix $Q_t$, reversible measure $\mu$, and starting from $k=0$. Concisely, the following chain of inequalities hold true in this context:
	\begin{align}\label{eq:hermon-peres-inequalities}
		\rho^Q_0\le T^Q_{2,0}\le \bar \rho^Q_0+5\, t_{\rm rel}^Q\le \rho_0^Q+8\, t^Q_{\rm LS}+10\, t^Q_{\rm rel}\ ,
	\end{align}
where:
\begin{itemize}
	\item $T^Q_{2,0}$ is the $L^2$-mixing time of the $Q$-chain when starting from $k=0$, i.e., 
	\begin{equation}
		T^Q_{2,0}:= \inf\left\{t\ge 0: \left\|Q_t\left(\frac{\IND_0}{\mu}\right)-1\right\|_{L^2(\mu)}\le \frac{1}{2}\right\}\ ;	
	\end{equation}
\item $t_{\rm rel}^Q:=(\lambda_{\rm gap}^Q)^{-1}$ is the relaxation time of the chain;
\item  $t_{\rm LS}^Q:=(\beta_{\rm LS}^Q)^{-1}$ is the inverse of the Log-Sobolev constant of the chain; 
\item letting $\mu([0,M)):=\sum_{\ell=0}^{M-1}\mu(\ell)$ and $\mu([M,d]):=\sum_{\ell=M}^d\mu(\ell)$, $M\in \{1,\ldots, d\}$,	
\begin{align}\label{eq:rho-0-bar}
	\bar \rho_0^Q\coloneqq\inf\left\{t\ge 0 : \Prob^Q_0\left(\tau_M>t\right)\le \mu([0,M))^3\ \text{for all}\ M\in \{1,\ldots, d/2\}\right\}\ ;
\end{align}
\item 
 $\rho_0^Q$ is defined analogously to $\bar \rho_0^Q$ with $\mu([0,M))^3$ replaced by \begin{equation}\mu([0,M))+\frac{1}{2}\sqrt{\mu([0,M))\,\mu([M,d])}\ .
\end{equation}
\end{itemize}

We now specialize the inequalities in \eqref{eq:hermon-peres-inequalities} to our setting of $P$- and $S$-chain. It is immediate to check by tensorization and comparison of Dirichlet forms that $t^P_{\rm rel},  t^P_{\rm LS}, t^S_{\rm rel}, t^S_{\rm LS}=O(1)$. 
Further,  we know the $L^2$-mixing of the Ehrenfest urn, namely $T^P_{2,0}\le \frac{1}{4}\log d+\Theta(1)$. These two observations yield
\begin{align}
	\rho_0^P\le \bar \rho_0^P \le \frac{1}{4}\log d+a\ ,
\end{align}for some $a>0$ large enough. In view of the definition of $\bar \rho_0^P$ (see \eqref{eq:rho-0-bar}), and 	 Lemma \ref{lemma:final},
we get, for all $M\in \{1,\ldots, d/2\}$,	
\begin{align}
	\Prob^S_0\left(\tau_M>\frac{1}{4}\log d+a+b\right)&\le \nu([0,M))^3+\Prob^P_0\left(\tau_M>\frac{1}{4}\log d+a\right)\\
	&\le 2\, \nu([0,M))^3\ ,
\end{align}
where $b=b(\sigma)>0$ is the constant appearing in the statement of Lemma \ref{lemma:final} for $\sigma=3$. Note that $ 2\, \nu([0,M))^3\le \nu([0,M))$ for all $M \in \{1,\ldots, d/2\}$. Therefore, 	the definition of $\rho_0^S$, the last inequality in \eqref{eq:hermon-peres-inequalities} and $t_{\rm rel}^S,t_{\rm LS}^S=O(1)$ yield the desired upper bound on $T^S_{2,0}$, thus, concluding the proof of the proposition. 
\end{proof}

\section{$L^1$-mixing on  complete bipartite graphs}\label{sec:complete-bipartite-l1}

In this section, we prove Theorem \ref{th:K1}. All throughout, we   let $m \in \{1,\ldots, \lfloor\frac{n}{2}\rfloor\}$, and   consider  the  complete bipartite graph with parts $|\cC_1|=m$ and  $|\cC_2|=n-m$. In this context,  $\cL$ in \eqref{eq:generator-avg} reads as follows:
\begin{equation}
	\cL f(\eta)= \sum_{x\in \cC_1}\sum_{y \in \cC_2}\left(f(\eta^{xy})-f(\eta)\right)\ ,\qquad \eta \in \Delta(V)\ , \ f:\Delta(V)\to \R\ .
\end{equation}

The proof of Theorem \ref{th:K1} is divided into several steps. Following \cite{chatterjee2020phase}, we use the discretized dynamics defined as follows.  
Define	
\begin{align}\label{eq:H}
H&\coloneqq \lfloor  \log_2 n - \log(n)^{1/3} \rfloor \  .	
\end{align}
We divide a unit mass into $2^H$ chunks. To each chunk we assign  a label $u \in \{1,\ldots, 2^H\}$, a position in $V$, and a mass equal to $2^{-H}$. 
We add a cemetery state  $\{\dagger\}$ and define a dynamics of these mass chunks as follows.
For every chunk $u=1,\ldots, 2^H$, we let $e_t(u) \in V\cup \{\dagger\}$ denote its position at time $t\ge 0$, and introduce two sequences  $(\beta_t(u))_{t\ge 0}$ and $(\alpha_t(u))_{t\ge 0} \subset \N_0$, with	  $\beta_0(u)= \alpha_0(u)=0$.
We write $(w_t(x))_{x\in V\cup \{\dagger\}}$ for the amount of mass at $x$  at time $t$, that is 
\begin{equation}
w_t(x)\coloneqq \sum_{u =1}^{2^H} 2^{-H} \IND_{e_t(u) =x}\ ,\qquad x \in V\cup \{\dagger\}\ ,\  t \ge 0\,.
\end{equation}
The dynamics is given by the following process:
\begin{enumerate}
\item all chunks are initially placed on the same site $x_0\in V$: $e_0(u)= x_0 , \;\forall u$;
\item the process follows the same edge updates of the averaging process: each edge rings independently with rate 1;
\item when edge $xy$ rings at time $t$, two types of scenarios can occur:
\begin{enumerate}
	\item good scenarios: 
	\begin{enumerate}
		\item there are no chunks at the endpoints of $xy$, and nothing occurs, or
		\item exactly  one of the endpoints, say $y$, of the edges is empty, and $x$ contains $2^i$ chunks with $i>0$; in this case half of the chunks on $x$ are selected uniformly at random and moved to $y$; the other half remains at $x$; concurrently, for all chunks $u$ sitting at either $x$ or $y$, we set $\alpha_t(u)\coloneqq \alpha_{t^-}(u)+1$; 	
	\end{enumerate}
	\item bad scenarios:
	\begin{enumerate}
		\item $\beta$-event: if both $w_{t^-}(x)>0$ and $w_{t^-}(y)>0$, then set   
		\begin{equation}\beta_{t'}(u)= 1\ ,\qquad e_{t'}(u)= \dagger\ ,
		\end{equation}
		for all $ t'\ge  t$, and all $u=1,\ldots, 2^H$ such that $e_{t^-}(u)\in \{x,y\}$. In particular, $w_t(x)=w_t(y)=0$, and $w_t(\dagger)=w_{t^-}(\dagger)+w_{t^-}(x)+w_{t^-}(y)$. 
		\item  $\alpha$-event: if exactly  one of the extremes, say $y$, of the edges is empty, and $x$ contains exactly one chunk $u$, i.e.\ $\alpha_{t^-}(u)=H$, then set, for all $t'\ge t$,
		\begin{equation}
			\alpha_{t'}(u)=H+1\ ,\qquad	e_{t'}(u)=\dagger\ .
		\end{equation}
	\end{enumerate}
\end{enumerate}
\end{enumerate}
In words, the discretized mass splits exactly along an edge if the proposed edge is not occupied at both ends, while if both ends are occupied all mass chunks from that edge are sent to the cemetery. Moreover, if a  single mass chunk sits at one of the endpoints of the proposed edge, then it is sent to the cemetery. 
By construction, the following properties hold a.s.:
\begin{enumerate}
\item the total mass of the chunks   is preserved by the dynamics, i.e.,
\begin{equation}
	w_t(\dagger)+	\sum_{x\in V} w_t(x)=1\ ,\qquad t \ge 0, 
\end{equation}
and the mass at $\dagger$ increases when a bad scenario occurs. 
\item if $(\eta_t)_{t\ge 0}$ denotes the Averaging process with all the mass initially placed at $x_0\in V$, then
\begin{equation}
	\eta_t(x)\ge w_t(x)\ ,\qquad x \in V\ ,\ t \ge 0\ ;
\end{equation}
\item the process $(w_t)_{t\ge 0}$ eventually sends all chunks to $\dagger$, i.e.,
\begin{equation}
	\lim_{t\to \infty} w_t(\dagger)=1\ .
\end{equation}
\item for any chunk not in $\dagger$ at time $t\ge 0$, the variable $\alpha$ satisfies 
\begin{equation}
	w_t(e_t(u)) = 2^{-\alpha_t(u)}\ ,\qquad \forall  u :\,\;  e_t(u)\neq\dagger \ .
\end{equation}
\end{enumerate}

We start with a technical lemma, and all throughout adopt the  following notation (cf.\ \eqref{eq:T-K1}):
\begin{gather}\label{eq:deftmixpm}
	t_{\rm mix}\coloneqq \frac{1}{2\gamma}\log_2 n\ , \qquad \gamma\coloneqq \frac{m(n-m)}{n}\ ,\\
	t^\l_{\rm mix}\coloneqq  t_{\rm mix}+\frac{\l}{\gamma} \sqrt{\log n}\ ,\qquad \l\in\R\ .
\end{gather}

\begin{lemma}\label{lem:technical1}
For every $\varepsilon>0$, there exist constants $\l=\l(\varepsilon)>0$ and $b=b(\varepsilon)>0$ such that for all $u=1,\ldots, 2^H$,
\begin{align}\label{eq:claim-technical-lemma}
\liminf_{n\to \infty} \Prob\left( \log_2(n)-b\sqrt{\log n} \le \alpha_{t^{-\l}_{\rm mix}}(u)\le H\,,\, \beta_{t^{-\l}_{\rm mix}}(u)=0 \right) \ge1-\varepsilon.
\end{align}
\end{lemma}
The proof of Lemma \ref{lem:technical1}  is postponed to the end of this section.
Next, we assume its validity and finish the proof of the theorem through the following two propositions.

\begin{proposition}[Lower bound]
For every $\varepsilon > 0$, there exists $\l=\l(\varepsilon)>0$ such that for every $x_0\in V$
\begin{align}
\liminf_{n\to \infty} \E_{\IND_{x_0}}\left[\left\|\frac{\eta_{t^{-\l}_{\rm mix}}}{\pi}-1 \right\|_{1} \right] \ge 2-\varepsilon\ .
\end{align}
\end{proposition}
\begin{proof}  In view of the definition and properties of the process $(w_t)_{t\ge 0}$, setting $t=t^{-\l}_{\rm mix}$,
\begin{align}
\frac{1}{2}\,\E_{\IND_{x_0}}\left[\left\|\frac{\eta_t}{\pi}-1 \right\|_{1} \right]
&= \sum_{x \in V}\E_{\IND_{x_0}}\left[\left(\eta_t(x)-\frac{1}{n}\right)_+ \right]\\
& \ge \sum_{x \in V}\E_{\IND_{x_0}}\left[\left(\eta_t(x)-\frac{1}{n}\right)_+ \IND_{\{w_t(x)\ge \frac{1}{n}\}} \right]\\
&\ \ge \sum_{x \in V}\E_{\IND_{x_0}}\left[\left(w_t(x)-\frac{1}{n}\right) \IND_{\{w_t(x)\ge \frac{1}{n}\}} \right]\\
& = 2^{-H}\sum_{u=1}^{2^H}\Prob\left(e_t(u)\neq\dagger\right)- \frac{1}
{n}\sum_{x\in V} \E_{\IND_{x_0}}\left[\IND_{\{w_t(x)\ge \frac{1}{n}\}}\right]\\
&  \ge \Prob\left(\alpha_t(u_0)\le H, \beta_t(u_0)=0\right) -\frac{2^H}n\ ,
\end{align}
where $u_0$ is any fixed chunk label, and we use  symmetry and the fact that at most $2^H$ sites can be non empty at any given time. Finally, since $2^H=o(n)$, Lemma \ref{lem:technical1} concludes the proof.	
\end{proof}

\begin{proposition}[Upper bound]
For every $\varepsilon > 0$, there exists $\l=\l(\varepsilon)>0$
such that for every $x_0\in V$
\begin{align}
\limsup_{n\to \infty} \E_{\IND_{x_0}}\left[\left\|\frac{\eta_{t^\l_{\rm mix}}}{\pi}-1 \right\|_{1} \right] \le \varepsilon\ .
\end{align}
\end{proposition}
\begin{proof}
Fix constants $c,b>0$ to be specified later, and let $T^-:=t_{\rm mix}^{-c}$ be defined as in \eqref{eq:deftmixpm}. Given the processes $(\eta_t)$ and $(w_t)$, we define a third process $(\tilde \eta_t)$ as follows:
\begin{enumerate}
\item for times $t< T^-$, $\tilde \eta_t = w_t$;
\item for times $t\ge  T^-$, we let $(\tilde \eta_t)$ evolve as an Averaging process, following the same edge updates of $(\eta_t)$, but starting at time $T^-$ from the following configuration (possibly, with total mass strictly less than one): for every $x\in V$,
\begin{equation}
	\tilde \eta_{T^-}(x)\coloneqq \begin{cases}
		w_{T^-}(x) &\text{if}\ w_{T^-}(x)\le \frac{2^{b\sqrt{\log n}}}{n}\\
		0 &\text{otherwise}\ 	;	\end{cases}
\end{equation}
\item define $|\tilde \eta|:= \sum_{x\in V}\tilde \eta_{T^-}(x)\in [0,1]$.
\end{enumerate}

In view of this construction, for all $t\geq T^-$,
\begin{align}
&	\E_{\IND_{x_0}}\left[\left\|\frac{\eta_t}{\pi}-1 \right\|_{1} \right]\\
& \le \sum_{x\in V}\E\left[\left|\eta_t(x)-\frac{1}{n}-\left(\tilde \eta_t(x)-\frac{|\tilde \eta|}{n}\right) \right| \right] +
\sum_{x\in V} \E\left[\left|\tilde \eta_t(x)-\frac{|\tilde \eta|}{n}\right| \right] \\
&\le  1-\E\left[|\tilde \eta| \right]+ \sum_{x\in V}\E\left[\eta_t(x)-\tilde \eta_t(x) \right]  + \E\left[|\tilde \eta|\left\|\frac{\tilde \eta_t}{|\tilde \eta|\pi}-1 \right\|_{1} \right]\ ,
\end{align}
where we use the triangle inequality twice   and the fact that $\eta_t(x)\ge \tilde \eta_t(x)$ a.s.. 
As a consequence,
\begin{align}\label{eq:ub-star}
\E\left[\left\|\frac{\eta_t}{\pi}-1 \right\|_{1} \right]\le 2\left(1-\E\left[ |\tilde\eta|\right] \right) + \E\left[|\tilde \eta|\left\|\frac{\tilde \eta_t}{|\tilde \eta|\pi}-1 \right\|_{1} \right]\ .
\end{align}
The first term on the right-hand side of \eqref{eq:ub-star} is at most $2\varepsilon$ if the constants $c=c(\varepsilon)$, where $c$ is the constant appearing in $T^-=t_{\rm mix}^{-c}$, and $b=b(\varepsilon)$ are large enough. Indeed,
by Lemma \ref{lem:technical1},
\begin{align}
\E\left[|\tilde \eta| \right]= 2^{-H}\sum_{u=1}^{2^H} \Prob\left( \log_2(n)-b\sqrt{\log n} \le \alpha_{T^-}(u)\le H\,,\, \beta_{T^-}(u)=0 \right) \ge 1-\varepsilon\,.
\end{align}
As for the second term on the right-hand side of \eqref{eq:ub-star}, letting $\cF_s$ denote the $\si$-algebra generated by the edge updates up to time $s$, the Markov property yields
\begin{align}
\E\left[|\tilde \eta|\left\| \frac{\tilde \eta_t}{|\tilde \eta|\pi}-1\right\|_{1} \right]&= \E\left[|\tilde \eta|\E\left[\left\| \frac{\tilde \eta_t}{|\tilde \eta|\pi}-1\right\|_{1} \bigg| \cF_{T^-}\right] \right]\\
&\le   \E\left[|\tilde \eta|\sqrt{\E\left[\left\| \frac{\tilde \eta_t}{|\tilde \eta|\pi}-1\right\|_{2}^2 \bigg| \cF_{T^-}\right]} \right]\\
&\le  \E\left[|\tilde \eta|\sqrt{e^{-(t-T^-)/t_{\rm rel}} \left\|\frac{\tilde \eta_{T^-}}{|\tilde \eta|\pi}-1 \right\|_{2}^2} \right]\ .
\end{align}
Above we used Cauchy-Schwarz inequality and Aldous-Lanoue's $L^2$-upper bound from time  $T^-$ to $t\geq T^-$.
By definition of the process $(\tilde \eta_t)$,
\begin{align}
\left\|\frac{\tilde \eta_{T^-}}{|\tilde \eta|\pi}-1 \right\|_{2} \le 	 \left\|\frac{\tilde \eta_{T^-}}{|\tilde \eta|\pi} \right\|_{2} 
\le \left\| \tilde \eta_{T^-}\right\|_{\infty} \frac{n}{|\tilde \eta|}\le \frac{2^{b\sqrt{\log n}}}{|\tilde \eta|}\,.
\end{align}
Finally, applying this with $t=t^\l_{\rm mix}$ and recalling that $t_{\rm rel} = 2/m$, 
\begin{align}
\E\left[|\tilde \eta|\left\| \frac{\tilde \eta_{t^\l_{\rm mix}}}{|\tilde \eta|\pi}-1\right\|_{1} \right]&\le \exp\left(\frac{-(t^\l_{\rm mix}-T^-)}{2t_{\rm rel}}\right) 2^{b\sqrt{\log n}}\,  \E\left[\sqrt{|\tilde \eta|}  \right]\\
&\le 
\exp\left(-\frac{n}{n-m}\frac{c+\l}{4}\sqrt{\log n}\right) 2^{b\sqrt{\log n}}\\
&\le \exp\left( \left(b\log 2-\frac{c+\l}{4}\right)\sqrt{\log n} \right).
\end{align}
For all $\varepsilon>0$ and $b,c>0$, by choosing $\l>0$ sufficiently large, the right-hand side is bounded above by $\varepsilon$ for all $n$ large enough. 
This concludes the proof.
\end{proof}

\begin{proof}[Proof of Lemma \ref{lem:technical1}]
Set $t:=t^{-c}_{\rm mix}$ and fix $u\in \{1,\ldots, 2^H\}$. 
We start by proving that, for every $c>0$,
\begin{equation}\label{eq:claim-beta}
\lim_{n\to \infty}	\Prob(\beta_t(u)=0)=1\ .	
\end{equation}
Suppose the chunk $u$ starts in $\cC_2$ and note that the hitting time $\tau_1$ of the set $\cC_1$ for $u$ is ${\rm Exp}(m/2)$. Once the chunk is in $\cC_1$ it will reach $\cC_2$ at some later time, and will visit again $\cC_1$ at some time $\tau_2$. Iterating, we define the successive arrival times to $\cC_1$ as $0<\tau_1<\tau_2<\dots$. 
Then, by neglecting the time spent in $\cC_1$ each time, we can estimate 
\begin{equation}\label{eq:bound-tau}
\Prob(\tau_j < t )\le \Prob\left(\sum_{i=1}^j X_i <t \right)\ ,\qquad j \in \N\ ,
\end{equation}
where $X_i$ are i.i.d.\ ${\rm Exp}(m/2)$. If instead $u$ starts in $\cC_1$ then $\tau_1=0$ and the  
inequality in \eqref{eq:bound-tau} holds with the sum on the right hand side going up to  $j-1$ instead of $j$. 
Therefore, recalling that $t=O(m ^{-1}\log n)$,  taking e.g.\ $j=\log^2(n)$ in \eqref{eq:bound-tau} we have
\begin{equation}\label{eq:bound-number_visits}
\Prob(u\ \text{hits $\cC_1$ more than $\log^2(n)$ times before $t$} )=o(1)\ .
\end{equation}
Let us consider the event $E_i$ that after the $i$-th hitting of $\cC_1$, $u$ experiences a $\b$-event before moving to $\cC_2$ or to $\dagger$ because of an $\a$-event.  We are going to show that
uniformly in $i\ge1$, \begin{equation}\label{eq:T-timeo}
\log^2(n)\Prob(E_i )=o(1).
\end{equation}
In view of \eqref{eq:bound-number_visits}, by a union bound this implies 
\begin{align}\label{eq:beta-event-on-center}
\Prob\left(\exists\, s\le t\ \text{s.t.}\ e_{s^-}(u)\in\cC_1\, ,\, \beta_{s}(u)=1\right)=o(1).
\end{align}
To prove \eqref{eq:T-timeo}, observe that, while in $\cC_1$, $u$ experiences attempts to average with other vertices at a rate $n-m$, and that at any attempt one has the following options:\begin{enumerate}
\item $u$ remains in $\cC_1$ with probability  $q_0\le 1/2$;
\item
$u$ moves to $\dagger$ because of a $\b$-event with probability $q_1\le 2^H/(n-m) \eqqcolon p=O( 2^{-\log^{1/3}n})$, since at any given time there are at most $2^H$ vertices with positive mass;
\item $u$ moves to $\cC_2$, or to $\dagger$ because of an $\a$-event, with probability $1- q_0-q_1\ge \frac13$. 
\end{enumerate}In particular, the overall number of attempts before moving to $\cC_2$ or to $\dagger$ is dominated by a geometric random variable with parameter $1/2$. 
Thus, a union bound shows that $\Prob(E_i)\le 6p= O( 2^{-\log^{1/3}n})$, which proves \eqref{eq:T-timeo}.

In order to prove \eqref{eq:claim-beta}, we are left with proving 
\begin{align}\label{eq:beta-event-on-c2}
\Prob\left(\exists\, s\le t\ \text{s.t.}\ e_{s^-}(u)\in\cC_2\, ,\, \beta_{s}(u)=1\right)=o(1).
\end{align}
To this end, observe that, while in $\cC_2$, $u$ experiences attempts to average with other vertices at a rate $m$, and thus by the argument in \eqref{eq:bound-number_visits} we may assume that there are at most $\log^2(n)$ such attempts up to time $t$.  Thus, letting $\Xi=\{0<\xi_1<\xi_2<\dots\}$ denote the set of random times at which these attempts occur, it suffices to show that uniformly in $s\le t$, 
\begin{align}\label{eq:beta-event-on-c22}
\log^2(n)\,	\Prob\left(e_{s^-}(u)\in\cC_2\, ,\, \beta_{s}(u)=1\,|\,s\in\Xi\right)=o(1).
\end{align}
To prove this, suppose $s\in\Xi$, say $s=\xi_i$ for some $i\in\bbN$. 
The event $e_{s^-}(u)\in\cC_2\, ,\, \beta_{s}(u)=1$ implies that the attempt involves $u$ and a vertex $x\in\cC_1$ with $w_s(x)>0$. Thus  
\begin{align}\label{eq:beta-event-on-c23}
&\Prob\left(e_{s^-}(u)\in\cC_2\, ,\, \beta_{s}(u)=1\,|\,s\in\Xi\right)\\
&\qquad = \sum_{x\in\cC_1} \Prob\left(\text{edge }(e_{s^-}(u),x)\text{ rings at time $s$}\,,\; w_{s^-}(x)>0\,|\,s\in\Xi\right)
\\
&\qquad = \frac1m\sum_{x\in\cC_1} \Prob\left(w_{s^-}(x)>0\,|\,s\in\Xi\,, \;\text{edge }(e_{s^-}(u),x)\text{ rings at time $s$}\,\right).
\end{align}
For any $a>0$, the event $w_{s^-}(x)>0$ implies that either $x$ has received mass within the time interval $[s-a,s]$, or $x$ has received no mass and experienced less than $H$ updates within the time interval $[s-a,s]$. Therefore,
\begin{align}\label{eq:beta-event-on-c24}
&\Prob\left(e_{s^-}(u)\in\cC_2\, ,\, \beta_{s}(u)=1\,|\,s\in\Xi\right)\\
&\quad \leq \max_{x\in\cC_1}\Prob\left(w_{s^-}(x)>0\,|\,s\in\Xi\,,\;\text{edge }(e_{s^-}(u),x)\text{ rings at time $s$}\right)\\&
\quad\le\max_{x\in\cC_1} \Prob\left(A_x\right) + \max_{x\in\cC_1}\Prob\left(B_x\right)\,,
\end{align}
where $A_x$ is the event that $x$ received mass within the time interval $[s^--a,s^-]$, and $B_x$ is the event that $x$ experienced less than $H$ updates within the time interval $[s^--a,s^-]$.
Recalling that there are at most $2^H$ vertices with positive mass,  one has
\begin{align}\label{eq:beta-event-on-c25}
\Prob\left(A_x\right) \le 2^H a\,.
\end{align}
On the other hand, the probability of less than $H$ updates at $x$ in $[s^--a,s^-]$ is bounded by the probability that a Poisson random variable with parameter $(n-m)a\ge na/2$ is less or equal than $H$. 
Thus,
\begin{align}\label{eq:beta-event-on-c26}
\Prob\left(B_x\right) \le \left(na/2\right)^{H+1} e^{- na/2}\,.
\end{align}
We can take $a= (H2^H)^{-1}$ so that $\Prob\left(A_x\right) =o(1)$ and $\Prob\left(B_x\right)\le 2^{\log ^2(n)}e^{-2^{\log ^{1/4}(n)}} = o(1)$.  This ends the proof of  \eqref{eq:claim-beta}.

\

To conclude the proof of the lemma we are left with estimating $\alpha_t(u)$, that is the number of good-scenario splittings involving the chunk labeled $u$ up to time $t$. 
We suppose $e_0(u)\in\cC_1$; the proof for the case $e_0(u)\in\cC_2$ is analogous.
We 
proceed as follows:
\begin{enumerate}
\item let $(Z_i)_{i\in \N}$ be i.i.d.\ ${\rm Geom}(1/2)$; each $Z_i \in \N$ indicates the number of attempts needed for  $u$ 
to succeed in moving to the other block;
\item given $(Z_i)_{i \in \N}$, let $( Y_{i,j})_{j=1}^{Z_i}$ be a  family of independent Exponential r.v.'s, with
\begin{equation}
	Y_{i,j} \sim \begin{cases}
		{\rm Exp}(n-m)&\text{if }i\text{ odd}\\
		{\rm Exp}(m)&\text{if }i\text{ even}\ ;
	\end{cases}
\end{equation} 
the r.v. $\sum_{j=1}^{Z_i} Y_{i,j}$ represents the  time $u$ spends  during the $i$-th visit to a  block before moving to the other block.
\end{enumerate}
Note that the random variable
\begin{equation}
K_t\coloneqq\sup\left\{k\ge 1\::\:\sum_{i=1}^k \sum_{j=1}^{Z_i}  Y_{i,j}\le t\right\}\:
\end{equation}
represents the number of switchings from a block to the other. Therefore, if $\beta_t(u)=0$, 
\begin{equation}
(H+1)\wedge 	\sum_{i=1}^{K_t} Z_i\le \alpha_t(u)\le \sum_{i=1}^{K_t+1} Z_i,
\end{equation}
where the $+1$ term in the second summation bounds the number of splittings that $u$ experiences in the time intercurring from its last change of block.

On the one hand, since for every $i\ge 1$
$$ \E\left[\sum_{j=1}^{Z_i}  Y_{i,j}+\sum_{j=1}^{Z_{i+1}}  Y_{i,j}\right]= \frac{2}{\gamma}\ ,$$
by the CLT, for every $\varepsilon>0$, there exists $\bar \zeta=\bar\zeta(\varepsilon)>0$ such that, for all $\zeta >\bar \zeta$ and for all $n$ large enough
\begin{align}\label{eq:ensu}
\Prob\left(\sum_{i=1}^{\gamma t+\zeta \sqrt{\log n}} \sum_{j=1}^{Z_i}  Y_{i,j}> t\, \right) \ge 1-\varepsilon\, .
\end{align}
Note that $\bar{\zeta}$ can be chosen independently of the constant $c$ in the definition of $t=t^{-c}_{\rm mix}$.
From \eqref{eq:ensu} it follows that 
\begin{equation}
\Prob\(K_t+1<\gamma t+\zeta \sqrt{\log n} \)\ge 1-\varepsilon\ .
\end{equation}
Hence,  
\begin{align}\label{eq:clt2}
\Prob\left(\alpha_t(u)\le H\,,\, \beta_t(u)=0 \right)
&\ge\Prob\left(\sum_{i=1}^{K_t+1} Z_i\le H\,,\, \beta_t(u)=0 \right)\\
&\ge \Prob\left(\sum_{i=1}^{K_t+1} Z_i\le H ,\: K_t+1\le \gamma t+\zeta \sqrt{\log n}  \right)-\varepsilon\,.
\end{align}
To finish the proof, observe that 
\begin{align}
\E\left[\sum_{i=1}^{\gamma t+\zeta \sqrt{\log n}}Z_i \right]&
= \log_2 n -2 \left(c-\zeta\right)\sqrt{\log n}\,.
\end{align}
Thus, 
by choosing $c=c(\zeta,\varepsilon)>0$ large enough, another application of the CLT shows that \eqref{eq:clt2} is at least $1-2\varepsilon$. 

For what concerns the lower bound on $\alpha_t(u)$, we proceed in a similar fashion. By the CLT there exists some $\bar\xi(\varepsilon)>0$ such that for all $\xi>\bar\xi(\varepsilon)$
\begin{align}
\Prob\left(\sum_{i=1}^{\gamma t-\xi \sqrt{\log n}} \sum_{j=1}^{Z_i}  Y_{i,j}\le t\, \right) \ge 1-\varepsilon\ ,\qquad 	\Prob\(K_t>\gamma t-\xi \sqrt{\log n} \)\ge 1-\varepsilon\ .
\end{align}
Therefore, for every $\varepsilon >0$ and $\xi > \bar \xi(\varepsilon)>0$,
\begin{align}
&\Prob\left(\alpha_t(u)> \log_2 n -b\sqrt{\log n}\,,\, \beta_t(u)=0 \right)\\
&\ge \Prob\left(\sum_{i=1}^{K_t}Z_i> \log_2 n -b\sqrt{\log n},\: \beta_t(u)=0 \right)\\
&\ge \Prob\left(\sum_{i=1}^{K_t}Z_i> \log_2 n -b\sqrt{\log n},\: K_t>\gamma t-\xi \sqrt{\log n} \right)-\varepsilon
\ge 1-2\varepsilon\ ,
\end{align}
where the last step follows by 
\begin{align}
\E\left[\sum_{i=1}^{\gamma t-\xi \sqrt{\log n}}Z_i \right]
&= \log_2 n -2 \left(\xi+c\right)\sqrt{\log n}\ ,
\end{align}
and the CLT by choosing $b=b(\varepsilon,c,\xi)>0$ large enough.
This completes the proof of Lemma \ref{lem:technical1}.
\end{proof}

\section{$L^2$-mixing on complete bipartite graphs}\label{sec:complete-bipartite-l2}
This section is devoted to the proof of Theorem \ref{th:K2}. We keep the same notation as 
 at the beginning of Section \ref{sec:complete-bipartite-l1}; in particular, $1\le m\le n/2$, and $|\cC_1|=m$, $|\cC_2|=n-m$. Furthermore, we remark that, in this context, $\Lrw$ of ${\rm RW}(G)$ reads as
\begin{equation}
	\Lrw \psi(x)=\frac{1}{2}\sum_{y\in V}\big(\psi(y)-\psi(x) \big)\left(\IND_{x\in \cC_1}\IND_{y\in \cC_2}+\IND_{x\in \cC_2}\IND_{y\in \cC_1}\right)\ ,\qquad x \in V\ ,\ 	 \psi :V\to \R \ ,
\end{equation}
from which it is straightforward to check that
\begin{equation}\label{eq:trel-complete-bipartite}
	t_{\rm rel}=\frac{2}{m}\ .
\end{equation}

Moreover, we remark that 
\begin{equation}\label{eq:worst-case-K2}
\text{the worst-case  initial condition is $\xi=\IND_{x_0}$ with $x_0\in \cC_2$}\ .
	\end{equation} Indeed,  observe that:
	\begin{itemize}
		\item starting with all the mass at a single site $x_0$ is worst-case by Proposition \ref{prop:worst-case}, and the dependence on $x_0$ is only through the subset $\cC_i$, $i=1,2$, such that $x_0\in\cC_i$;
		\item starting from $\eta=\IND_{x_0}$, $x\in V$, after the first effective update for ${\rm Avg}(G)$, the next configuration is $\frac{1}{2}\IND_{x_0}+\frac{1}{2}\IND_y$, for some $y\in V$ in $\cC_i\not\ni x_0$; further,  from that time on, mixing will not depend on neither $x_0$ nor $y$;
		\item  letting $\tau_i$ denote the random time of the first effective update when $x_0 \in \cC_i$,  
		\begin{align}
			\tau_i \sim \begin{cases}
				{\rm Exp}(n-m) &\text{if}\ i=1\\
				{\rm Exp}(m) &\text{if}\ i=2\ ;
			\end{cases} 
		\end{align}
		Hence, since $n-m\ge m$, such times can be coupled such that, a.s.,  $\tau_1\le \tau_2$.
	\end{itemize}

Before entering the details of the proof of the theorem (Sect.\ \ref{sec:proof-th:K2} below),  let us first discuss the particularly relevant case of the star graph, i.e., $G=K_{m,n-m}$ with $m=1$.
\subsection{A motivating example: the star graph}
 Set $m=1$,  and let $x_*$ denote the unique element in $\cC_1$ (i.e., the center of the star). As pointed out in Section \ref{suse:strategy-bipartite}, the $L^2$-mixing must occur between times $\frac{t_{\rm rel}}{2}\log n$ and $t_{\rm rel}\log n$. As shown in the following corollary, which specializes Theorem \ref{th:K2} to the setting $m=1$, the $L^2$-cutoff  occurs strictly in between these two extremes, namely around times $\frac{3}{4}t_{\rm rel}\log n$.

\begin{corollary}[$L^2$-cutoff on the star graph]\label{pr:star-l2}
	Let $T(a):= \frac{3}{4}t_{\rm rel}\left(\log n+ a \right)$, $a \in \R$. Then, 
\begin{equation}\label{eq:prop-star-l2-leaf}
	\lim_{n\to \infty}\sup_{\xi\in \Delta(V)} \E_{\xi}\left[\left\|\frac{\eta_{T(a)}}{\pi}-1 \right\|_{2}^2 \right] = e^{-a}\ ,\qquad a \in \R\ .
\end{equation}
\end{corollary}
Let us briefly comment over this result, with the aim of illustrating that the arguments used to derive  Proposition \ref{prop:AL}, while they provide a sharp $L^2$-estimate for the case of the complete graph $K_n$, they do not prove to be useful in the context of the star graph.

A simple computation shows that
\begin{align}
\frac{\dd}{\dd t}	\E_\xi\left[\left\|\frac{\eta_t}{\pi}-1 \right\|_{2}^2 \right]&=- \E_\xi\left[\cErw\left(\frac{\eta_t}{\pi} \right) \right] \\
&= - \frac{1}{t_{\rm rel}} \left( \E_\xi\left[\left\| \frac{\eta_t}{\pi}-1 \right\|_{2}^2 \right] +  \E_\xi\left[\left(\frac{\eta_t}{\pi}(x_*)-1 \right)^2 \right]\right),
\end{align}
 for all $\xi \in \Delta(V)$ and $t\ge 0$.
By integrating over time and setting $\xi=\IND_x$ for some $x\neq x_*$, 
\begin{equation}
	\label{eq:star-l2-identity}
	\E_{\IND_x}\left[\left\|\frac{\eta_t}{\pi}-1 \right\|_{2}^2 \right]= e^{-\frac{t}{t_{\rm rel}}} \left\|\frac{\IND_x}{\pi}-1 \right\|_{2}^2 - e^{-\frac{t}{t_{\rm rel}}} \int_0^t \frac{e^{\frac{s}{t_{\rm rel}}}}{t_{\rm rel}}\, \E_{\IND_x}\left[\left(\frac{\eta_s}{\pi}(x_*)-1 \right)^2 \right]\dd s
\ .
\end{equation}
Since $\left\|\frac{\IND_x}{\pi}-1 \right\|_2^2\approx n$,  the first term on the right-hand side above is divergent for $t=(1-\delta) t_{\rm rel}\log n$,  $\delta \in (0,1)$, and  drops down to order $\Theta(1)$ only if $t= t_{\rm rel}\left(\log n+a\right)$. Therefore,  $L^2$-mixing of ${\rm Avg}(G)$ at the strictly smaller time $T(a)=\frac{3}{4}t_{\rm rel}\left(\log n+a\right)$ must be due to cancellations of diverging terms between  first and second terms on the right-hand side of \eqref{eq:star-l2-identity}.
In particular, in view of the duality relation in \eqref{eq:duality-bin2}, controlling this latter term boils down to estimating the transition probabilities of ${\rm CRW}(G)$ in the time interval  $[0,T(a)]$.  This approach proves to be difficult and, instead of dealing with ${\rm CRW}(G)$, one might be tempted to bound the last term in \eqref{eq:star-l2-identity} by relying on the duality with the simpler process  {${\rm RW}(G)$}. In order to do so, one must use Jensen inequality, obtaining
	\begin{equation}\label{eq:lb}
		e^{-\frac{t}{t_{\rm rel}}} \int_0^t \frac{e^{\frac{s}{t_{\rm rel}}}}{t_{\rm rel}}\, \E_{\IND_x}\left[\left(\frac{\eta_s}{\pi}(x_*)-1 \right)^2 \right]\dd s\ge
		e^{-\frac{t}{t_{\rm rel}}}\int_0^t \frac{e^{\frac{s}{t_{\rm rel}}}}{t_{\rm rel}}\left(\E_{\IND_x}\left[\frac{\eta_s}{\pi}-1\right] \right)^2\dd s\ .
	\end{equation}
Nonetheless, it turns out that  the right-hand side of \eqref{eq:lb} is of order $O(1)$ for all $t\ge 0$, implying that \eqref{eq:star-l2-identity} combined with the lower bound in \eqref{eq:lb} is too weak to deduce the desired control. 

In the proof of Theorem \ref{th:K2} we show how to handle the transition probabilities of ${\rm CRW}(G)$  around times $T(a)$ to derive first order estimates on the $L^2$-distance without relying directly on \eqref{eq:star-l2-identity}. 

\subsection{Proof of Theorem \ref{th:K2}}\label{sec:proof-th:K2}
We now turn to the proof of Theorem \ref{th:K2} in the general case $m\in \{1,\ldots,\lfloor\frac{n}{2}\rfloor\}$. In preparation of this, we introduce
\begin{equation}\label{eq:Bb}
b \in [0,\tfrac{1}{2}]\longmapsto 	B(b):=\sqrt{9-32 b+32 b^2} \in [1,3]\ ,
\end{equation}
and note that the function $B(b)$  is non-increasing, as well as  $B(0)=3$ and $B(\frac{1}{2})=1$. Finally, we recall the definitions of $\theta=\theta_{m,n-m}$, $T(a)$, and $D(b)$ from the statement of  Theorem \ref{th:K2}, and present three remarks.

\begin{remark}[Asymptotics of $\theta$]\label{rem:theta}
	By a first-order Taylor expansion, 
	\begin{align}
		\lim_{n\to\infty}\theta_{m,n-m}&=\begin{dcases}\lim_{n\to\infty}\frac{3n}{8m}\left(1-\left(1-\frac{16}{9}\frac{m}{n}\left(1+o(1)\right)\right)\right)=\frac{2}{3} &\text{if}\  m/n\to 0\ ,\\
			\frac{3}{8b}
			\left(
			1-\sqrt{1-\frac{32}{9}b(1-b)}
			\right)\in [\tfrac{1}{2},\tfrac{2}{3}) &\text{if}\ m/n\to b\in (0,\tfrac{1}{2}]\ .
		\end{dcases}
	\end{align}
\end{remark}
\begin{remark}[$T(a)$ in terms of $t_{\rm rel}$]\label{rem:cutoff-trel}
	 Recall from \eqref{eq:trel-complete-bipartite} that $t_{\rm rel}=\frac{2}{m}$, and \eqref{eq:Bb}. Define
	\begin{align}
		b\in (0,\tfrac{1}{2}]\longmapsto 	C(b):=	\frac{4b}{3-B}\in (\tfrac{3}{4},1]\ ,
	\end{align}	
	and note that $C(b)$ is increasing on $(0,\frac{1}{2})$, as well as $C(0):=\lim_{b\downarrow 0} C(b)=\frac{3}{4}$, and $C(\frac{1}{2})=1$. Then, assuming that $\frac{m}{n}\to b\in [0,\frac{1}{2}]$, 	 $T(a)$ appearing in Theorem \ref{th:K2} can be rewritten in terms of the relaxation time using the following asymptotic equality: for all $a \in \R$,
	\begin{equation}\label{eq:tmix-as-fuc-of-trel}
		T(a):=\frac{1}{\theta m}\left(\log n+a\right)\sim 
		C(b)\,t_{\rm rel} \left(\log n+a\right)\ .
	\end{equation}

\end{remark}

\begin{remark}[Shape of the cutoff profile]\label{rem:profile}
The cutoff profile at time $T(a)$ is exponential in $a\in \R$, and depends on $b=\lim_{n\to \infty}m/n\in [0,1/2]$ through   $D(b)$ in \eqref{eq:Db}, which exhibits the following behavior:
	\begin{align}
		D(b)\in [0,\tfrac{\sqrt 6-2}{4}]\ ,\qquad \lim_{b\downarrow 0}D(b)=D(0)=D(\tfrac{1}{2})=0\ ,\qquad D(\tfrac{3}{8})=	\tfrac{\sqrt 6-2}{4}\ .
	\end{align}
This term $D(b)$ does not appear when considering $L^2$-mixing for ${\rm RW}(G)$ on $G=K_{m,n-m}$.
\end{remark}

\begin{proof}[Proof of Theorem \ref{th:K2}]
	Fix $x_0\in V$ and $t>0$;	 then, by duality in \eqref{eq:duality-bin2},
	\begin{align}\label{eq:l2-norm-bipartite}
		\begin{aligned}
			&	\E_{\IND_{x_0}}\left[\left\|\frac{\eta_t}{\pi}-1 \right\|_{2}^2 \right]= n\sum_{x\in V}\E_{\IND_{x_0}}\left[\left(\eta_t(x)\right)^2 \right]-1\\
			&= n\sum_{x\in\cC_1}\Purw_{x_0,x_0}\left(X_t=x,Y_t=x\right)+ n\sum_{y\in \cC_2} \Purw_{x_0,x_0}\left(X_t=y,Y_t=y\right)-1\ .
		\end{aligned}
	\end{align}
	In what follows, we analyze the above transition probabilities for ${\rm CRW}(G)$ by lumping some states and considering the following auxiliary Markov chain with five states
	\begin{align}
		0_1&:=\{(x,x):x \in \cC_1\}\\
		0_2&:=\{(y,y): y\in \cC_2\}\\
		1&:= \{(x,y)\cup (y,x): x\in \cC_1, y \in \cC_2\}\\
		2_1&:= \{(x,x'): x, x'\in \cC_1,x\neq x'\}\\
		2_2&:= \{(y,y'): y, y'\in \cC_2,y\neq y'\}\ ,
	\end{align}
	and corresponding $5\times 5$ transition rate matrix
	\begin{align}\label{eq:Q-rate-matrix}
		Q=\begin{pmatrix}
			-\frac{3(n-m)}{4} &\frac{n-m}{4} &\frac{n-m}{2} &0 &0\\
			\frac{m}{4} &-\frac{3m}{4} &\frac{m}{2} &0 &0\\
			\frac{1}{4} &\frac{1}{4} &-\frac{n-1}{2}&\frac{m-1}{2} &\frac{n-m-1}{2}\\
			0&0&n-m&-(n-m)&0
			\\
			0&0&m&0&-m
		\end{pmatrix}\ .
	\end{align}
and reversible measure $\mu:\{0_1,0_2,1,2_1,2_2\}\to [0,1]$ given by
\begin{align}\label{eq:nu}
	\mu=n^{-2}\begin{pmatrix}
		m&n-m& 2m(n-m)&m(m-1)&(n-m)(n-m-1)
	\end{pmatrix}\ .
\end{align}
This lumping allows us to rewrite  \eqref{eq:l2-norm-bipartite} as follows:
	\begin{align}
		\E_{\IND_{x_0}}\left[\left\|\frac{\eta_t}{\pi}-1 \right\|_{2}^2 \right]=
		n\, e^{tQ}(0_i,0_1)+ n\, e^{tQ}(0_i,0_2)-1\ ,\qquad x_0\in \cC_i\ ,\ i \in \{1,2\}\ .
	\end{align}
  Let $\{(\rho_j,\phi_j)\}_{j=0}^4$ denote the eigenvalue-eigenfunction pairs of $-Q$, with $(\rho_0=0, \phi_0\equiv1)$, and $\{\phi_j\}_{j=0}^4$ orthonormal basis in $L^2(\mu)$; hence, we have
	\begin{align}\label{eq:l2-norm-complete-bipartite}
		\begin{aligned}
		&\E_{\IND_{x_0}}\left[\left\|\frac{\eta_t}{\pi}-1 \right\|_{2}^2 \right]
		=
		n\, e^{tQ}(0_i,0_1)+ n\, e^{tQ}(0_i,0_2)-1\\
		&\qquad= n\, \mu(0_1)\sum_{j=0}^4 e^{-\rho_j t}  \phi_j(0_i)\, \phi_j(0_1)+n\, \mu(0_2)\sum_{j=0}^4 e^{-\rho_j t}  \phi_j(0_i)\, \phi_j(0_2)-1\\
		&\qquad= \sum_{j=1}^4e^{-\rho_j t}\phi_j(0_i)\left(\frac{m}{n}\phi_j(0_1)+\frac{n-m}{n}\phi_j(0_2)\right)\ ,\qquad x_0 \in \cC_i\ ,\ i \in \{1,2\}\ .
		\end{aligned}
	\end{align}
Furthermore, the following eigenpair $(\rho_2,\phi_2)$ for $-Q$ is explicit:
\begin{align}\label{eq:tilde-lambda-4}
	\rho_2=\frac{n}{2}\ ,\qquad \phi_2	&= \sqrt{ \frac{1}{n-m}}\begin{pmatrix}-\frac{n-m}{m} & 1 &\frac{n-2m}{2m}&-\frac{n-m}{m} &1
	\end{pmatrix}\ .
\end{align}
 Although the remaining eigenvalues $\rho_1, \rho_3, \rho_4$ and eigenfunctions $\phi_1, \phi_3, \phi_4$ of $-Q$ are, in general, not explicit, we now provide sharp asymptotic estimates which 
 will yield the desired claim. 

Recall from \eqref{eq:worst-case-K2} that, in order to analyze the worst-case mixing, it suffices to consider $x_0\in \cC_2$. Hence, with such a choice, \eqref{eq:l2-norm-complete-bipartite} reads as
\begin{align}\label{eq:expression-L2-norm}
	\E_{\IND_{x_0}}\left[\left\|\frac{\eta_t}{\pi}-1 \right\|_{2}^2 \right]= \sum_{j=1}^4e^{-\rho_j t}\, \Psi_j\ ,
\end{align} 
with 
\begin{align}\label{eq:Phij}
	\Psi_j:=\phi_j(0_2)\left(\frac{m}{n}\phi_j(0_1)+\frac{n-m}{n}\phi_j(0_2)\right)\ ,\qquad j=1,\ldots, 4\ .
\end{align}
Recall that $\lim_{n\to \infty}\frac{m}{n}=b\in [0,\frac{1}{2}]$. In view of \eqref{eq:expression-L2-norm}, we conclude the proof of the theorem as soon as
 we show that, for $t=T(a)$, $a \in \R$, as given  in \eqref{eq:T-K2},  the following three claims hold, as $n\to \infty$:
\begin{align}\label{eq:first-eigenvalue-eigenfunction-asymptotics}
	\Psi_1=\left(1+D(b)+o(1)\right)n\ ,\qquad 	e^{-\rho_1 t}\,n = \left(1+o(1)\right) e^{-a}\ ,
\end{align}
and
\begin{equation}\label{eq:vanishing-j234}
	\sum_{j=2}^4 e^{-\rho_jt}\,\Psi_j=o(1)\ .
\end{equation}
 Actually, in order to have the latter claim in   \eqref{eq:vanishing-j234}, it will suffice to prove 
 \begin{equation}\label{eq:large-evals}
 	\min\left\{\rho_2,\rho_3,\rho_4 \right\}\ge \frac{2}{5}\, n\ .
 \end{equation}
Indeed, since $\mu(0_2)=\frac{n-m}{n^2}$ and $\|\phi_j\|_{L^2(\mu)}=1$, we get, for all $j=0,\ldots, 4$,
\begin{align}
	|\phi_j(0_2)|\le \frac{n}{\sqrt{n-m}}\ ,
\end{align} as well as, recalling that $\mu(0_1)=\frac{m}{n^2}$,
\begin{align}
\Psi_j:=	\frac{m}{n}\phi_j(0_1)+\frac{n-m}{n}\phi_j(0_2)\le \sqrt{\frac{m}{n}\left(\phi_j(0_1)\right)^2+\frac{n-m}{n}\left(\phi_j(0_2)\right)^2}\le \sqrt n\ .
\end{align}
By combining these two estimates with that in \eqref{eq:large-evals}, we get
\begin{align}\label{eq:j234}
	&\left|	\sum_{j=2}^4 e^{-\rho_jt}\,\Psi_j\right|
	\le 3 \sqrt 2\,n\, e^{-\frac{2}{5}nt}=o(1)\ ,
\end{align}
where for the last step we used the expression of $t=T(a)=\frac{1}{\theta m}\left(\log n+a\right) $ and $ \frac{2}{5}\frac{n}{\theta m}>1$.

We are now left with proving \eqref{eq:large-evals} and the two claims in \eqref{eq:first-eigenvalue-eigenfunction-asymptotics} about approximating the quantity $\Psi_1$ and the eigenvalue $\rho_1$.
Recall the definition of $Q$ in \eqref{eq:Q-rate-matrix}, of its reversible measure $\mu$ in $\eqref{eq:nu}$ and $\{(\rho_j,\phi_j)\}_{j=0}^4$ of its eigenpairs; then, 
letting $U:=\diag(\sqrt \mu)$, we introduce
\begin{align}
	W:= \tfrac{1}{n}\,U(-Q)U^{-1}-(U\phi_0)(U\phi_0)^\top\ .
\end{align} 
It is immediate to check that $W$ is symmetric and that its eigenpairs $\{(\rho_j^W,\varphi^W_j)\}_{j=0}^4$ are related to those of $-Q$ as follows: $\varphi_j^W=U \phi_j$ for all $j=0,\ldots, 4$, while
\begin{align}\label{eq:eigenvalues-T}
	\rho_0^W=-1\neq 0=\rho_0\ ,\qquad \rho_j^W=\tfrac{1}{n}\rho_j\ ,\quad j=1,\ldots, 4\ .
\end{align}
Furthermore, it is simple to check that $W$ decomposes as a sum $W=S+R$, with $S$ being a block-diagonal matrix and $R=O(n^{-1/2})$, i.e., all entries of $R$ are of order $O(n^{-1/2})$. More precisely, letting $b=b_n:=\frac{m}{n}\in (0,\frac{1}{2}]$, $S$ reads as
\begin{align}
	S=\left(\begin{array}{c|c}
	S_0 & 0\\
	\hline
	0&S_1
	\end{array} \right)\ ,\qquad \text{with}\quad S_0 := \left(\begin{array}{cc}
	\frac{3}{4} \left(1-b\right) & -\frac{1}{4} \sqrt{b
	\left(1-b\right)}\\
	-\frac{1}{4} \sqrt{b \left(1-b\right)} & \frac{3}{4}b
	\end{array}\right)\ ,
\end{align}
and
\begin{align}
	S_1:=\left(
	\begin{array}{ccc}
		 \frac{\left(1-2b\right)^2}{2} & -\frac{\sqrt{b\left(1-b\right)
		}\left(1+2b\right)}{\sqrt{2}} & -\frac{\sqrt{b
				\left(1-b\right)}\left(3-2b\right)}{\sqrt{2}} \\
	-\frac{\sqrt{b\left(1-b\right)
		}\left(1+2b\right)}{\sqrt{2}} & 1-b\left(1+b\right) & -b\left(1-b\right) \\
	 -\frac{\sqrt{b \left(1-b\right)}\left(3-2b\right)}{\sqrt{2}}
		&- b\left(1-b\right) & b\left(3-b\right)-1 \\
	\end{array}
	\right)\ .
\end{align}
The eigenvalues $\{\rho_j^S\}_{j=0}^4$ of $S$ are then the union of those of $S_0$ and $S_1$, which  are explicit:
\begin{align}
	{\rm spec}(S_0)=\left\{\frac{3-B}{8},\frac{3+B}{8}\right\}\ ,\qquad {\rm spec}(S_1)=\left\{-1,\frac{1}{2},1\right\}\ ,
\end{align}
where $B=B(b)$ was defined in \eqref{eq:Bb}.	In particular, the smallest eigenvalue $\rho_0^S$ equals $-1$,  the second smallest $\rho_1^S$ is $\frac{3-B}{8}$, and, since $B\in [1,3]$, $\rho_j^S\ge \frac{1}{2}$ for all $j=2,\ldots, 4$. This simple observation implies the following two facts:
\begin{enumerate}
	\item on the one side, since $W=S+R$,  $R=O(n^{-1/2})$ and \eqref{eq:eigenvalues-T}, Weyl theorem for eigenvalues of sums of Hermitian matrices (see, e.g., \cite[Th.\ 4.3.1]{horn_matrix_2012}) yields 
	\begin{align}\label{eq:weyl}
		\left|\tfrac{1}{n}\rho_j- \rho_j^S\right| =  \left|\rho_j^W-\rho_j^S \right|= O(n^{-1/2})\ ,\qquad j=1,\ldots, 4\ ,
	\end{align} 
and, in turn, \eqref{eq:large-evals};
\item on the other side, we obtain $	\delta_1^S:=\min_{j\neq 1}|\rho_1^S-\rho_j^S|\ge \frac{1}{4}$.
\end{enumerate} 
By employing the latter claim and $R=O(n^{-1/2})$, we estimate the eigenfunction $\phi_1$ of $-Q$ as follows. By applying Davis-Kahan  theorem \cite{davis_kahan_rotation_1970} to the $L^2(1)$-normalized eigenfunctions $\varphi_1^W=U\phi_1$ and $\varphi_1^S$ associated to the second-smallest eigenvalues of $W$ and $S$, respectively, we obtain 
\begin{equation}\label{eq:davis-kahan}
	\|U\phi_1-\varphi_1^S \|_{L^2(1)}=\|\varphi_1^W-\varphi_1^S\|_{L^2(1)} \le\frac{\|R\|_{L^2(1)}}{\delta_1^S}=O(n^{-\frac{1}{2}})\ .
\end{equation}
Then, due to the simple form of $S$ and $S_0$, the $L^2(1)$-normalized eigenfunction $\varphi_1^S$ is explicit and given by $\varphi_1^S=\frac{1}{Z}(A,1,0,0,0)$, where \begin{equation}
	A=A(b):=\frac{6b-3+B}{2\sqrt{b\left(1-b\right)}}\ ,\qquad Z:=\sqrt{A^2+1}\ .
\end{equation} 
Note that if $\frac{m}{n}=b=o(1)$, then 
$A(b)=(1+o(1))\frac13 \sqrt{b}$, while $A(b)\in (0,\infty)$ if $b\in (0,\frac{1}{2}]$.
By combining  \eqref{eq:davis-kahan} with the definitions of $\varphi_1^S$,  $U:=\diag(\sqrt n)$ and $\Psi_1$ (see \eqref{eq:Phij}), 	
 we get
\begin{equation}\label{eq:Phi1-general}
	\begin{split}
		\Psi_1
		&=\left(1+o(1)\right)n\left(\frac{A}{A^2+1}\frac{\sqrt{b}}{\sqrt{1-b}}+\frac{1}{A^2+1} \right)
		=n\left(1+D+o(1)\right)\ ,
	\end{split}
\end{equation}
where $D=D(b)$ was given in \eqref{eq:Db}; this corresponds to the first claim in \eqref{eq:first-eigenvalue-eigenfunction-asymptotics}.

In order to prove the second claim in \eqref{eq:first-eigenvalue-eigenfunction-asymptotics} for \emph{all} values of $b=\frac{m}{n}$, we need the following  estimate for $\rho_1$:
\begin{align}\label{eq:eigenvalue-first-sharp}
	\rho_1=m  \left(\theta+O\left(\frac{1}{n}\right)\right)\ ,
\end{align}
which is, in general, sharper than the one we obtained in \eqref{eq:weyl} above.	
To this aim, we analyze the characteristic polynomial $\lambda\mapsto p(\lambda)$ of $-Q$ divided by $\left(\lambda-\rho_0\right)\left(\lambda-\rho_4\right)=\lambda\left(\lambda-\frac{n}{2}\right)$:
\begin{equation}\label{eq:def-q}
	q(\lambda):=\frac{p(\lambda)}{\lambda\left(\lambda-\frac{n}{2}\right)}
	= -\lambda^3 + \frac{7n-2}{4}\lambda^2+\frac{2m^2+(1-2m)n-3n^2}{4}\lambda+\frac{mn\left( n-m\right)}{2}\ ,
\end{equation}
and show that, within  a window of order $O(\frac{m}{n})$ around $ \theta m$, 	$q(\lambda)$ changes sign; since $\frac{m}{n}\ll\theta m < \frac{1}{3} n$,  $\min\{\rho_2,\rho_3,\rho_4\}>\frac{2}{5}n$ (see \eqref{eq:large-evals}) and $\rho_0=0$, 	this would ensure the validity of \eqref{eq:eigenvalue-first-sharp}.
Recall the definition of $\theta$ in \eqref{eq:def-rho}; then, writing $\frac{m}{n}=b\in (0,\frac{1}{2}]$, for any $ c\in\R$, we have
\begin{equation}
	q\left(bn \left(\theta+\frac{c}{n}\right)\right)=- \frac{1+o(1)}{64}\left(3-B+b\left(18c+10B-16\right)+b^2\left(16-64c \right)+64c b^3  \right)n^2 \ .
\end{equation}
Recall that $B\in [1,3)$ if $b\in (0,\frac{1}{2}]$, while $B(b)=3-\frac{16}{3}b+O(b^2)$ if $b=o(1)$. In either case, by choosing $c=c(b)\in \R$ large (resp.\ small) enough, the expression in parenthesis above becomes positive (resp.\ negative). This concludes the proof of the theorem.  
	\end{proof}

\appendix
\section{Proof of  (\ref{eq:eigenvaluessss2})}\label{app:est-hardy}
Recall that $\nu(j):=2^{-d}\binom{d}{j}$, $j\in \{0,1,\ldots, d\}$, and that we want to prove
\begin{align}\label{eq:eigenvaluess2}
	\limsup_{d\to \infty}\sup_{1\le M\le d/2}\left(\lambda^P_{M,0}\right)^{-1} \log\left(\frac{1}{ \nu([0,M))}\right)<\infty\ .
\end{align}
We divide the proof into two steps.

\subsection{Discrete weighted Hardy's inequality}

We start by estimating $\lambda_{M,0}^P$ by means of a discrete weighted Hardy's inequality (see, e.g., \cite{miclo_example_1999}).
\begin{lemma}\label{lemma:hardy}
	For every $M\in \{1,\ldots, d/2\}$, define
	\begin{align}\label{eq:CM}
		C_M:= \max_{0\le k\le M-1}\nu([0,k])\sum_{j=k}^{M-1}\frac{1}{\nu(j)(d-j)}\ .
	\end{align}
	Then, 
	\begin{align}
		\frac{1}{4C_M}\le	\lambda_{M,0}^P\le 	\frac{1}{C_M}\ .
	\end{align}
\end{lemma}	
\begin{proof}
Recall that $\lambda_{M,0}^P> 0$ is the largest value $\lambda>0$ satisfying
\begin{align}
	\lambda\,	\nu(f^2)\le 	 \cE_P(f)\  ,\qquad f \in \cF_{M-1}:= \left\{f :\{0,1,\ldots, d\}\to \R: 
	f=0\ \text{on}\ [M,d] 
	\right\}\ .
\end{align}
We now apply the following version of weighted Hardy's inequality (e.g., \cite{miclo_example_1999} or \cite[Prop.\ A.3]{chen_saloff-coste_mixing_bd_chains_2013}): for all positive measures $\pi, \mu$ on $\{0,1,\ldots, M-1\}$, define
\begin{align}
	\varLambda=\varLambda_{\pi,\mu}:= \frac{1}{4}\left(\max_{0\le k\le M-1} \pi([0,k])\sum_{j=k}^{M-1}\frac{1}{\mu(j)}\right)^{-1}\ ;
\end{align}
then
\begin{align}
	\varLambda	\sum_{k=0}^{M-1}\left(\sum_{j=k}^{M-1}g(j)\right)^2 \pi(k)\le 	 \sum_{k=0}^{M-1} \left(g(k)\right)^2 \mu(k)\ ,\qquad g:\{0,1,\ldots, M-1\}\to \R\ .
\end{align}
Applying this to $\pi(k):= \nu(k)$, $\mu(k):= \nu(k)(d-k)$, we get, setting, for every $f\in \cF_{M-1}$,  $g(k):= f(k)-f(k+1)$, \begin{equation}4 \varLambda\ge 	\lambda_{M,0}^P\ge \varLambda\ .
\end{equation}
This concludes the proof of the lemma.
\end{proof}

In view of Lemma \ref{lemma:hardy}, \eqref{eq:eigenvaluess2} is actually equivalent to
\begin{align}\label{eq:final_eigenvalue_CM}
\limsup_{d\to \infty}\sup_{1\le M\le d/2} C_M\, \log \left(\frac{1}{\nu([0,M))}\right)<\infty\ ,
\end{align}
where $C_M>0$ is given in \eqref{eq:CM}.
The rest of this section is devoted to the proof of  \eqref{eq:final_eigenvalue_CM}.
As a first step, by straightforward monotonicity arguments, 	we obtain the following upper bounds:
\begin{align}
&\sup_{1\le 	M\le d/2} C_M \log\left(\frac{1}{\nu([0,M))}\right)\\
& = \sup_{1\le M\le d/2}\sup_{0\le k\le M-1} \sum_{j=k}^{M-1}\frac{\nu([0,k])}{\nu(j)(d-j)}\left(-\log\nu([0,M-1])\right)\\	
&\le\sup_{1\le M\le d/2}\sup_{0\le k\le M-1}
\frac{1}{d-M+1} \sum_{j=k}^{M-1}\frac{\nu([0,k])}{\nu(j)}\left(-\log\nu([0,M-1])\right)\\
&\le\frac{3}{d}\sup_{0\le k<d/2}\nu([0,k])\sup_{k+1\le M\le d/2} \left(\sum_{j=k}^{M-1}\frac{1}{\nu(j)}\right)\left(-\log\nu([0,M-1])\right)\\
&\le 3 \sup_{0\le k <d/2} \frac{1}{d}\, \nu([0,k])\log\left(\frac{1}{\nu([0,k])}\right)\left(\sum_{j=k}^{d/2-1}\frac{1}{\nu(j)}\right)\ .
\end{align}
We have just showed that, by defining
\begin{align}
\varGamma(k)=\varGamma_d(k):=\frac{1}{d}\, \nu([0,k])\log\left(\frac{1}{\nu([0,k])}\right)\left(\sum_{j=k}^{d/2-1}\frac{1}{\nu(j)}\right)\ ,\qquad 0\le k <d/2\ ,
\end{align}
we have
\begin{align}\label{eq:final}
\sup_{1\le M\le d/2}C_M \log\left(\frac{1}{\nu([0,M))}\right)\le 3\, \sup_{0\le k<d/2} \varGamma(k)\ .
\end{align}

\subsection{Estimating $\varGamma$}
Recall from, e.g.,  \cite{alfers_dinges_normal_1984} the following representation for the binomial distribution:
\begin{align}\label{eq:alfers-dinges}
\nu(k):=\,2^{-d}\binom{d}{k}=\frac{1}{\sqrt{2\pi d}}\frac{\exp\left(-d H(\tfrac{k}{d})+\cS^{(d)}(\tfrac{k}{d})\right)}{\sqrt{\tfrac{k}{d}\left(1-\frac{k}{d}\right)}}\ ,\qquad 1\le k\le d/2-1\ ,
\end{align}
where
\begin{align}
H(p)=H(p||\tfrac{1}{2})&:= p\log\frac{p}{1/2}+(1-p)\log\frac{1-p}{1/2}\ ,\qquad p \in [0,1]\ ,
\end{align}
and $\cS^{(d)}:(0,1)\to \R$ is a function whose explicit form is unimportant for our purposes, yet  satisfies the following bounds:
\begin{align}
\cS^{(d)}(p):= \cS(d)-\cS(pd)-\cS((1-p)d)\ ,\qquad \cS(m)\in\left(-\frac{m^{-3}}{360},\frac{m^{-1}}{12} \right)\ .
\end{align}
In particular,  there exist constants $s_1<0<s_2$ such that
\begin{align}\label{eq:Sdkd}
\cS^{(d)}(\tfrac{k}{d})\in \left(s_1,s_2\right)\ ,\qquad 1\le k<d/2\ .
\end{align}

It is now straightforward	 to check that $\nu$ satisfies the  \emph{convexity hypothesis}, shortly	 ${\rm CONV}(c,\bar d)$, in \cite[App.\ A]{cancrini_martinelli_roberto_logarithmic_2002}, with $\gamma=\nu$, $n_{\rm min}=0 $, $n_{\rm max}=d$, 	 $\bar n= \lceil d/2\rceil$, and some  $c=c(s_1,s_2)\ge2$, therein. More precisely,	
\begin{enumerate}
\item $
c^{-1}\,d/2\le d-d/2\le c\,d/2$ is trivial;
\item  by monotonicity of $\nu$ on $\{0,1,\ldots, d/2\}$, 
\begin{align}\label{eq:conv-1}
	\frac{\nu(k-1)}{\nu(k)}\le 1\le  c \exp\left(-\frac{d/2- k}{c\,	 (d/2)}\right)\ ,\qquad 1 \le k \le d/2\ ;
\end{align}
\item by \eqref{eq:alfers-dinges}, \eqref{eq:Sdkd}, and	
\begin{align}
	2\left(\tfrac{1}{2}-p\right)^2\le	H(p)\le 4 \log 2 \left(\tfrac{1}{2}-p\right)^2\ ,\qquad p\in [0,1]\ ,
\end{align}
we have, for all $d\in \N$ large enough and $0\le k<d/2$,
\begin{align}
	\frac{1}{c\sqrt{d/2}} \exp\left(-2 c d \left(\frac{1}{2}-\frac{k}{d}\right)^2\right)\le \nu(k)\le \frac{1}{c\sqrt{d/2}}\exp\left(-\frac{2d}{c}\left(\frac{1}{2}-\frac{k}{d}\right)^2\right)\ .	
\end{align}
\end{enumerate}
As a consequence of these three conditions and \cite[Lem.\ A.2--4]{cancrini_martinelli_roberto_logarithmic_2002}, we have, for some $C>0$ independent of $d$, and all $d$ large enough,
\begin{equation}
\nu([0,k])
\le C\, \frac{d/2}{d/2-k}\, \nu(k)\ ,\qquad 0\le k<d/2\ ,
\end{equation}
\begin{equation}
\sum_{j=k}^{d/2-1}\frac{1}{\nu(j)}\le C\, \frac{d/2}{d/2-k}\, \frac{1}{\nu(k)}\ ,\qquad 0\le k<d/2\ ,
\end{equation}
and
\begin{equation}
\log\left(\frac{1}{\nu([0,k])}\right) \le  C\left(1+d\left(\frac{1}{2}-\frac{k}{d}\right)^2\right)\ ,\qquad 0\le k\le d/2-\sqrt{d/2}\ .
\end{equation}
Finally, combining these estimates as done in the proof of \cite[Prop.\ A.5]{cancrini_martinelli_roberto_logarithmic_2002} yields
\begin{align}
\limsup_{d\to \infty}\sup_{0\le k<d/2} \varGamma(k)<\infty\ ,
\end{align}
thus, by \eqref{eq:final}, \eqref{eq:final_eigenvalue_CM}. This concludes the proof.

\subsubsection*{Acknowledgments}
F.S.\ wishes to thank Universit\`{a} Roma Tre, Dipartimento di Matematica e Fisica,  for the very kind hospitality during an early stage of this work.

M.Q.\ thanks the German Research Foundation (project number 444084038, priority program SPP2265) for financial support.  F.S.\ gratefully acknowledges funding by
the Lise Meitner fellowship, Austrian Science Fund (FWF): M3211.

P.C.\ thanks the Miller Institute for Basic Research in Science for funding his visit to  UC Berkeley during the Fall 2022.

\end{document}